\documentclass[11pt]{article}
\date{}

\usepackage[title]{appendix}
\usepackage{xcolor}
\usepackage[margin=1in]{geometry}                
\usepackage{graphicx}
\usepackage{subcaption}
\usepackage{pdflscape}
\usepackage{amssymb}
\usepackage[normalem]{ulem}
\usepackage{hyperref}
\usepackage{enumitem}
\usepackage{mathrsfs}
\usepackage{epstopdf}
\usepackage{rotating}
\usepackage{longtable} 
\usepackage{adjustbox}
\usepackage{float}

\usepackage{color}
\usepackage{bbm, dsfont}
\usepackage{pst-node}
\usepackage{tikz-cd}
\usepackage{amsfonts} 
\usepackage{geometry}
\usepackage{amsthm}
\usepackage{amsmath}
\usepackage{titlesec}
\usepackage{amssymb}
\usepackage{enumitem}
\usepackage{float}
\usepackage [english]{babel}
\usepackage [autostyle, english = american]{csquotes}
\usepackage{algorithm}
\usepackage[noend]{algpseudocode} 
\makeatletter
\def\BState{\State\hskip-\ALG@thistlm}
\makeatother
\usepackage{hyperref}
\usepackage[normalem]{ulem}

\newlist{casess}{enumerate}{1}
\setlist[casess]{label=     \textbf{Case} \arabic*:}
\usepackage{mathtools}
\DeclarePairedDelimiter\ceil{\lceil}{\rceil}

\makeatletter
\newcommand*{\rom}[1]{\expandafter\@slowromancap\romannumeral #1@}
\makeatother

\usepackage{etoolbox}

\makeatletter
\patchcmd{\ttlh@hang}{\parindent\z@}{\parindent\z@\leavevmode}{}{}
\patchcmd{\ttlh@hang}{\noindent}{}{}{}
\makeatother

\usepackage{listings}
\usepackage{color} 
\definecolor{mygreen}{RGB}{28,172,0} 
\definecolor{mylilas}{RGB}{170,55,241}

\newlist{Assumptions}{enumerate}{1}
\setlist[Assumptions]{label=     \textbf{Assumption} \arabic*:}

\makeatletter

\newsavebox{\@brx}
\newcommand{\llangle}[1][]{\savebox{\@brx}{\(\m@th{#1\langle}\)}%
  \mathopen{\copy\@brx\kern-0.5\wd\@brx\usebox{\@brx}}}
\newcommand{\rrangle}[1][]{\savebox{\@brx}{\(\m@th{#1\rangle}\)}%
  \mathclose{\copy\@brx\kern-0.5\wd\@brx\usebox{\@brx}}}
\makeatother

\usepackage{lipsum} 
\usepackage{titlesec}
\titleformat{\subsection}[runin]
       {\normalfont\bfseries}
       {\thesubsection}
       {0.5em}
       {}
       [.]

\newcommand{\A}{\mathfrak{A}}

\newcommand{\B}{\mathfrak{B}} 

\newcommand{\CC}{\mathbb{C}}

\usepackage{tikz-cd}
\usepackage{graphicx}

\def\N{\mathbb{N}}

\def\R{\mathbb{R}}

\def\e{{\sf e}}

\def\BofH{\mathbb B(\mathcal H)}

\def\d{{\rm d}}
\def\bu{\bullet}

\def\({\left(}
\def\[{\left[}
\def\){\right)}
\def\]{\right]}

\def\si{\sigma}

\def\G{{\sf G}}

\def\p{\parallel}
\def\<{\langle}
\def\>{\rangle}
\providecommand{\norm}[1]{\lVert#1\rVert}

 \newtheorem{thm}{Theorem}[section]
 \newtheorem{cor}[thm]{Corollary}
 
 \newtheorem{lem}[thm]{Lemma}
 \newtheorem{prop}[thm]{Proposition}
 \theoremstyle{definition}
 \newtheorem{defn}[thm]{Definition}
 \theoremstyle{remark}
 \newtheorem{rem}[thm]{Remark}
 \newtheorem{ex}[thm]{Example}
 \numberwithin{equation}{section}

\numberwithin{equation}{section}

\usepackage[backend=biber, maxnames=10]{biblatex}
\begin{document}


\title{Polynomial growth and functional calculus in algebras of integrable cross-sections}

\author{Felipe I. Flores
}

\maketitle


\begin{abstract}
Let $\G$ be a locally compact group with polynomial growth of order $d$, a polynomial weight $\nu$ on $\G$ and a Fell bundle $\mathscr C\overset{q}{\to}\G$. We study the Banach $^*$-algebras $L^1(\G\,\vert\,\mathscr C)$ and $L^{1,\nu}(\G\,\vert\,\mathscr C)$, consisting of integrable cross-sections with respect to $\d x$ and $\nu(x)\d x$, respectively. By exploring new relations between the $L^p$-norms and the norm of the Hilbert $C^*$-module $L^2_\e(\G\,\vert\,\mathscr C)$, we are able to show that the growth of the self-adjoint, compactly supported, continuous cross-sections is polynomial. More precisely, they satisfy $$\|{e^{it\Phi}}\|=O(|t|^n),\quad\textup{ as }|t|\to\infty,$$ for values of $n$ that only depend on $d$ and the weight $\nu$. We use this fact to develop a smooth functional calculus for such elements. We also give some sufficient conditions for these algebras to be symmetric. As consequences, we show that these algebras are locally regular, $^*$-regular and have the Wiener property (when symmetric), among other results. Our results are already new for convolution algebras associated with $C^*$-dynamical systems. 
\end{abstract}


\section{Introduction}\label{introduction}

A great deal of effort within the subject of abstract harmonic analysis has been put into understanding the $L^1$-algebras and weighted $L^1$-algebras of locally compact groups. Most of this research has focused on extending properties of $L^1(\mathbb R)$ or $\ell^1(\mathbb Z)$ to more general algebras. In general, one could say that the base case is that of abelian groups and that case is fairly well understood (see \cite{Ru90}).

It seems natural to us that a continuation of these studies would consider group algebras twisted by a cocycle $L_\omega^1(\G)$, convolution algebras $L_\alpha^1(\G,\A)$ associated with a $C^*$-dynamical system $(\G,\A,\alpha)$, or more generally the algebra of cross-sections $L^1(\G\,\vert\,\mathscr C)$ associated with a Fell bundle $\mathscr C\overset{q}{\to}\G$. In fact, many authors have already taken this step over the years (cf. \cite{AuRa23,Le75,LP79,LeNg06}) and, in this article, we try to provide a unified approach that maximizes generality by choosing Fell bundle algebras. These algebras still have the abstract harmonic analysis flavor of a group algebra (and we hope that this article serves as proof of that), but they can also present some radically different behavior. Let us also mention that the technical difficulties we overcame here were already present at the level of $C^*$-dynamical systems, and so this choice does not overtly complicate too many things. 

Our main result is the construction of a functional calculus based on smooth functions for a large (that is, dense) subset of self-adjoint elements in $L^1(\G\,\vert\,\mathscr C)$, provided that $\G$ is a group of polynomial growth. This is done by estimating the growth of the $1$-parameter groups these elements generate and constructing a smooth functional calculus \`a la Dixmier-Baillet. 

In fact and for the purpose of obtaining interesting applications, we introduce the algebra $L^{1,\nu}(\G\,\vert\,\mathscr C)$, which denotes the dense subalgebra of cross-sections that are integrable with respect to $\nu(x)dx$ ($\nu$ being a weight on the group $\G$). The appearance of this algebra is relevant since many of the properties we wish to achieve for $L^{1}(\G\,\vert\,\mathscr C)$ are simply not true, but they can still be achieved in this dense subalgebra, provided that the necessary conditions on $\nu$ are met.

\begin{thm}
    Let $\G$ be a locally compact group with polynomial growth of order $d$. Let $\mathfrak D$ denote either $L^{1}(\G\,\vert\,\mathscr C)$ or $L^{1,\nu}(\G\,\vert\,\mathscr C)$, provided that there exists a polynomial weight $\nu$ on $\G$ such that $\nu^{-1}$ belongs to $L^p(\G)$, for some $0<p<\infty$. Let $\widetilde{\mathfrak D}$ be the minimal unitization of ${\mathfrak D}$. Then, there exists an $n\in\N$, dependent only on $d$ and the weight $\nu$, so that every $\Phi=\Phi^*\in C_{\rm c}(\G\,\vert\,\mathscr C)$ has the growth 
    $$
    \norm{e^{it\Phi}}_{\widetilde{\mathfrak D}}=O(|t|^n),\quad\textup{ as }|t|\to\infty.
    $$ 
    In particular, for any complex function $f\in C_{\rm c}^\infty(\R)$, with Fourier transform $\widehat f$, we have \begin{enumerate}
        \item[(a)] $f(\Phi)=\frac{1}{2\pi}\int_{\mathbb R} \widehat{f}(t)e^{it\Phi}\d t$ exists in $\widetilde{\mathfrak D}$, or in ${\mathfrak D}$ if $f(0)=0$.
        \item[(b)] For any non-degenerate $^*$-representation $\Pi:{\mathfrak D}\to\BofH$, we have 
        $$
        \widetilde{\Pi}(f\big(\Phi\big))=f\big(\Pi(\Phi)\big).
        $$
        Furthermore, if $\Pi$ is injective, we also have 
        $$
        {\rm Spec}_{\widetilde{\mathfrak D}}\big(f(\Phi)\big)={\rm Spec}_{\BofH}\big(\widetilde\Pi\big(f(\Phi)\big)\big).
        $$
        \end{enumerate}
\end{thm}

Let us mention that a similar calculus was constructed for group algebras in \cite{Di60} and for convolution algebras associated to finite-dimensional $C^*$-dynamical systems in \cite{LeNg061}. Variations have also appeared for weighted group algebras (see for example \cite{DLM04}). Functional calculi based on smooth functions have also been constructed in some other different (but related) contexts, such as \cite{KS94,Mo22}. The construction carried out in \cite{LeNg061} was impossible to generalize unaltered, and this led the authors to essentially abandon the study of the classical $L^1$-algebras in the follow-up work \cite{LeNg06}. In the present work, we demonstrate how to extend their results by adopting a very indirect strategy and via the use of a `strengthened' Young's inequality. 

Indeed, in the next lemma, we make use of the following norms (they will be re-introduced later in the text): 
\begin{align*}
    \norm{\Phi}_{L^\infty(\G\,\vert\,\mathscr C)}&= {\rm essup}_{x\in \G}\norm{\Phi(x)}_{\mathfrak C_x} ,\\
    \norm{\Phi}_{L^2(\G\,\vert\,\mathscr C)}&= \Big(\int_\G \norm{\Phi(x)}_{\mathfrak C_x}^2\,\d x\Big)^{1/2},\\
    \norm{\Phi}_{L^2_\e(\G\,\vert\,\mathscr C)}&= \norm{\int_\G\Phi(x)^\bu\bu\Phi(x)\,\d x}_{\mathfrak C_\e}^{1/2}.
\end{align*} 
It is fairly clear that $\norm{\cdot}_{L^2_\e(\G\,\vert\,\mathscr C)}\leq \norm{\cdot}_{L^2(\G\,\vert\,\mathscr C)}$, which makes the next inequality stronger than the usual Young's inequality (obtained by replacing $\norm{\cdot}_{L^2_\e(\G\,\vert\,\mathscr C)}$ with $\norm{\cdot}_{L^2(\G\,\vert\,\mathscr C)}$), hence justifying the name.

\begin{lem}[Strengthened Young's inequality]
     Let $\Phi,\Psi\in C_{\rm c}(\G\,\vert\,\mathscr C)$. Then
    $$\norm{\Psi*\Phi}_{L^\infty(\G\,\vert\,\mathscr C)}\leq \norm{\Psi}_{L^2(\G\,\vert\,\mathscr C)}\norm{\Phi}_{L^2_\e(\G\,\vert\,\mathscr C)}.$$
\end{lem}

The above inequality helps with the major difficulty that appeared before, namely the lack of appropriate norm estimates. The usual Young's inequality relates the $L^p$-norms, but the $C^*$-completion of $L^1(\G\,\vert\,\mathscr C)$, here denoted ${\rm C^*}(\G\,\vert\,\mathscr C)$, is represented in the Hilbert $C^*$-module $L^2_\e(\G\,\vert\,\mathscr C)$, which differs from the more obvious $L^2$-space, $L^2(\G\,\vert\,\mathscr C)$. This means that the classical Young's inequality cannot be used to infer anything about the operator norm and $C^*$-theory becomes unavailable. The introduction of new norms (and therefore new normed algebras) in \cite{LeNg06} responded to this necessity. Our inequality restores the communication between the $L^p$ norms and the operator norm and thus allows us to obtain results inside the realm of generalized $L^1$-algebras.

We now shift our attention to the many applications this functional calculus has. We will use our functional calculus to study properties of spectral nature like symmetry \cite{LP79,Lu79} or norm-controlled inversion \cite{GK13,GK14}, ideal-theoretic properties like $^*$-regularity \cite{Ba83,Bo80,LeNg04} and representation-theoretic properties like the Wiener property \cite{HJLP76,Le75,Lu79,FGLL03}. 

To be more precise, a reduced Banach $^*$-algebra is said to \begin{itemize}
    \item[(i)] be symmetric if it is closed under the holomorphic functional calculus of its universal $C^*$-completion,
    \item[(ii)] admit norm-controlled inversion if the norm of the inverse elements can be controlled using the universal $C^*$-norm,
    \item[(iii)] be $^*$-regular if its $^*$-structure space is that of its universal $C^*$-completion (or, more precisely, the canonical map is an homeomorphism),
    \item[(iv)] have Wiener property (inspired by the very classical Wiener's tauberian theorem \cite{Wi32}) if every closed, proper ideal is annihilated by a topologically irreducible $^*$-representation.
\end{itemize}

In the setting of groups with polynomial growth, we first mention the notable symmetry result of Losert \cite{Lo01}. The ideal theory of group algebras has been studied on \cite{Ba81,Lu80} (mostly under the assumption of symmetry) and the case of weighted group algebras in \cite{DLM04,FGLL03,MM98}. Norm-controlled inversion was studied for group algebras in \cite{Ni99,SaSh19}. Let us also mention that a $^*$-regular algebra has a unique $C^*$-norm and the question of uniqueness of $C^*$-norms in group algebras (or even group rings) has been studied in \cite{AlKy19,Bo84}.

For generalized convolution algebras, a lot of progress has been made in the study of symmetry \cite{FJM,Ku79,Le73,LP79}, but results about the Wiener property or $^*$-regularity of these algebras are few and far apart \cite{Le11,LeNg061}. In particular, results beyond the algebras of finite dimensional $C^*$-dynamical systems seem to be scarce. The only notable exception is \cite[Theorem 1]{Ki82}, which proves the $^*$-regularity of algebras associated with $C^*$-dynamical systems, provided that the acting group is abelian. In that context, we proved the following, far-reaching theorem.

\begin{thm}\label{main2}
    Let $\G$ be a locally compact group with polynomial growth of order $d$. Let $\mathfrak D$ denote either $L^{1}(\G\,\vert\,\mathscr C)$ or $L^{1,\nu}(\G\,\vert\,\mathscr C)$, provided that there exists a polynomial weight $\nu$ on $\G$ such that $\nu^{-1}$ belongs to $L^p(\G)$, for some $0<p<\infty$. Then $\mathfrak D$ is $^*$-regular. Moreover, if $\mathfrak D$ is also symmetric, then it has the Wiener property and for every unital Banach algebra $\B$ and each continuous unital homomorphism $\varphi:\widetilde{\mathfrak D}\to\B$, we have \begin{enumerate}
        \item[(a)] If $\Phi \in \widetilde{\mathfrak D}$ is normal, $${\rm Spec}_{\widetilde{\mathfrak D}/I}(\Phi+I)={\rm Spec}_{\B}\big(\varphi(\Phi)\big).$$ 
        \item[(b)] For a general $\Phi \in \widetilde{\mathfrak D}$, $${\rm Spec}_{\widetilde{\mathfrak D}/I}(\Phi+I)={\rm Spec}_{\B}\big(\varphi(\Phi)\big)\cup \overline{{\rm Spec}_{\B}\big(\varphi(\Phi^*)\big)},$$
    \end{enumerate} where $I={\rm ker}\,\varphi$.
\end{thm}

Here the utility of $L^{1,\nu}(\G\,\vert\,\mathscr C)$ is exemplified best. The existence of such $\nu$ is automatic if the group is compactly generated or when $\G$ can be saturated by an increasing sequence of compact subgroups (Examples \ref{construction}, \ref{locallyfinite}) and in such cases, $L^{1,\nu}(\G\,\vert\,\mathscr C)$ is automatically symmetric (Theorem \ref{closedness}), guaranteeing the conclusion of Theorem \ref{main2} in a dense subalgebra of $L^{1}(\G\,\vert\,\mathscr C)$. Furthermore, $L^{1}(\G\,\vert\,\mathscr C)$ is symmetric when $\G$ is nilpotent \cite{FJM} or compact (Theorem \ref{corcompact}), but there are examples of locally finite groups with nonsymmetric algebras (cf. \cite{Py82}), so the restriction to the weighted subalgebra is necessary in this case.

Lastly, we also proved the following result about norm-controlled inversion. It seems to be the first result known for convolution algebras where the algebra of coefficients is not $\CC$. 

\begin{thm}\label{main1}
   Let $\G$ be a locally compact group with polynomial growth of order $d$. Let also $\nu$ be a polynomial weight on $\G$ such that $\nu^{-1}$ belongs to $L^p(\G)$, for some $0<p<\infty$ and let $\mathfrak E=L^{1,\nu}(\G\,\vert\,\mathscr C)\cap L^\infty(\G\,\vert\,\mathscr C)$. Then the following are true. \begin{enumerate}
        \item[(i)] $\mathfrak E$ is a symmetric Banach $^*$-subalgebra of $L^{1}(\G\,\vert\,\mathscr C)$.
        \item[(ii)] If $\G$ is discrete and $\mathfrak E=\ell^{1,\nu}(\G\,\vert\,\mathscr C)$ has a unit, then it admits norm-controlled inversion.
    \end{enumerate}
\end{thm}

We now briefly describe the content of the article. Section \ref{prem} contains preliminaries. We mostly fix notation and prove the Dixmier-Baillet theorem, which we use to construct our functional calculus. Section \ref{mainsec} starts with us proving Lemma \ref{lemma}, which is our main computational lemma and contains the strengthened Young's inequality. The rest of the section is then devoted to deriving the growth estimate of the $1$-parameter groups $e^{it\Phi}$ in the case of $L^{1}(\G\,\vert\,\mathscr C)$ and constructing the Dixmier-Baillet calculus. In Section \ref{symmsub} we do the same but for $L^{1,\nu}(\G\,\vert\,\mathscr C)$, after showing that this algebra is both symmetric and inverse-closed in $L^{1}(\G\,\vert\,\mathscr C)$. In Section \ref{consequences} we derive our applications and Theorem \ref{main2}, based on the results we obtained and some well-established methods. Each part of the theorem is treated in a different subsection. In these subsections we also derive some other (smaller) properties of these algebras, not mentioned here. Subsection \ref{normcontrolled} is somewhat different from the rest, as it does not rely on the functional calculus we developed, but on the norm estimates. It is devoted to introducing the algebra $\mathfrak E$ and proving Theorem \ref{main1}. Finally, Section \ref{spectralradius} is an appendix, where we use Lemma \ref{lemma} to show that $C_{\rm c}(\G\,\vert\,\mathscr C)$ is quasi-symmetric in the subexponential growth case and that compact groups are hypersymmetric. This last fact both extends the applicability of our main theorem and helps us provide the first known examples of hypersymmetric groups with non-symmetric discretizations.

\section{Preliminaries}\label{prem}

Let $\B$ be a Banach algebra. $\B(b_1,\ldots, b_n)$ denotes the closed subalgebra of $\B$ generated by the elements $b_1,\ldots, b_n\in\B$. The set of invertible elements in $\B$ is denoted by ${\rm Inv}(\B)$. If $\B$ has an involution, $\B_{\rm sa}$ denotes the set of self-adjoint elements in $\B$. The Banach algebra of bounded operators over the Banach space $\mathcal X$ is denoted by $\mathbb B(\mathcal X)$. If $\B$ is a commutative Banach algebra with spectrum $\Delta$, then $\hat b\in C_0(\Delta)$ denotes the Gelfand transform of $b\in \B$. In particular, the Fourier transform of a complex function $f:\mathbb R\to\CC$ is $$\widehat{f}(t)=\int_{\mathbb R}f(x)e^{itx}\d x.$$

\begin{defn}\label{unitization}
    Let $\B$ be a Banach $^*$-algebra. If $\B$ is unital, we set $\widetilde{\B}=\B$. Otherwise, $\widetilde{\B}=\B\oplus \CC$ is the smallest unitization of $\B$, endowed with the norm $\norm{b+r1}_{\widetilde{\B}}=\norm{b}_{{\B}}+|r|$.
\end{defn}

\begin{rem}
    If $\Pi:\B\to\BofH$ is a non-degenerate $^*$-representation, then it extends naturally to a non-degenerate $^*$-representation $\widetilde\Pi:\widetilde\B\to\BofH$, defined by $\widetilde\Pi(b+r1)=\Pi(b)+r{\rm id}_{\mathcal H}$.
\end{rem}

As usual, ${\rm Spec}_{\B}(b)=\{\lambda\in\CC\mid b-\lambda1\textup{ is not invertible in }\widetilde\B\}$ will denote the spectrum of an element $b\in\B$, while $$\rho_\B(b)=\sup\{|\lambda|\mid \lambda\in {\rm Spec}_{\B}(b) \}$$ denotes its spectral radius. Gelfand's formula for the spectral radius says that $$\rho_\B(b)=\lim_{n\to\infty} \norm{b^{n}}_{\B}^{1/n}.$$

\begin{defn}
    Let $\B$ be a Banach $^*$-algebra. An element $b\in\B$ is said to have \emph{polynomial growth of order $d$} if $$\norm{e^{itb}}_{\widetilde\B}=O(|t|^d),\quad\textup{ as }|t|\to\infty.$$
\end{defn}

\begin{rem}
    Barnes \cite{Ba89} defines and studies this property in the context of the algebra $\mathbb B(\mathcal X)$ of operators on the Banach space $\mathcal X$. In particular, he provides many examples of operators with this property. Barnes' approach is equivalent to ours, as we can consider any Banach $^*$-algebra $\B$ as represented on $\widetilde{\B}$ via left multiplication.
\end{rem}

The key idea for us is that the self-adjoint elements of polynomial growth admit a smooth functional calculus. The concrete fact is the following theorem, attributed to Dixmier \cite[Lemme 7]{Di60} and Baillet \cite[Théorème 1]{MB79}. We include its proof here for two reasons: for the convenience of the reader and because it lies at the core of our argument.

\begin{thm}\label{DixBai}
    Let $\B$ be a Banach $^*$-algebra, and let $b\in\B$ be a self-adjoint element with polynomial growth of order $d$. Let $f:\mathbb R\to \mathbb C$ be a complex function that admits $d+2$ continuous and integrable derivatives. Let $\widehat{f}$ be the Fourier transform of $f$. Then the following is true.\begin{enumerate}
        \item[(i)] The following Bochner integral exists in $\widetilde{\B}(b,1)$: \begin{equation*}
            f(b)=\frac{1}{2\pi}\int_{\mathbb R} \widehat{f}(t)e^{itb}\d t.
        \end{equation*} Moreover, if $f(0)=0$, then $f(b)\in \B(b)$.
        \item[(ii)] For any non-degenerate $^*$-representation $\Pi:\widetilde{\B}(b,1)\to\BofH$, we have ${\Pi}\big(f(b)\big)=f\big(\Pi(b)\big)$. 
        \item[(iii)] ${\rm Spec}_{\widetilde{\B}(b,1)}\big(f(b)\big)=f\big({\rm Spec}_{\widetilde{\B}(b,1)}(b)\big)$.
        \item[(iv)] If $\B(b,1)$ is semisimple, then there is a $^*$-homomorphism $\varphi_b:C_{\rm c}^{d+2}(\mathbb R)\to \widetilde{\B}(b,1)$ defined by $\varphi_b(f)=f(b)$,  such that $\varphi_b(g)=1$ if $g\equiv 1$ on a neighborhood of ${\rm Spec}_{\widetilde{\B}(b,1)}(b)$ and $\varphi_b(g)=b$ if $g(x)=x$ on a neighborhood of ${\rm Spec}_{\widetilde{\B}(b,1)}(b)$.
        \item[(v)] For any $R>0$, there exists $M>0$ such that $\norm{\varphi_b(f)}_{\widetilde{\B}}\leq M\sum_{k=0}^{d+2}\norm{f^{(k)}}_{C_{0}(\mathbb R)}$, for all $f$ with support contained in $[-R,R]$. 
    \end{enumerate}
\end{thm}

\begin{proof} \emph{(i)} As $b$ has polynomial growth of order $d$, there is a constant $C>0$ such that 
$$
\norm{e^{itb}}_{\widetilde\B}\leq C(1+|t|)^d,\quad \textup{ for all }t\in\mathbb R,
$$ 
while the hypothesis on $f$ implies that 
$$
|\widehat{f}(t)|\leq \frac{A}{(1+|t|)^{d+2}},
$$ 
for some $A>0$. Then 
\begin{equation*}
    \int_{\mathbb R} \norm{\widehat{f}(t)e^{itb}}_{\widetilde\B}\d t \leq CA\int_{\mathbb R} \frac{(1+|t|)^d}{(1+|t|)^{d+2}}\d t=CA\int_{\mathbb R} \frac{1}{(1+|t|)^{2}}\d t=2CA, 
\end{equation*} so $f(b)\in \widetilde\B$. If $f(0)=0$, then $0=\int_{\mathbb R} \widehat{f}(t)\d t$ so $$f(b)=\frac{1}{2\pi}\int_{\mathbb R} \widehat{f}(t)(e^{itb}-1)\d t\in\B,$$ since $e^{itb}-1\in\B.$ 

\emph{(ii)} Given such a representation $\Pi$, we have 
\begin{equation}\label{calculation}
    \Pi\big(f(b)\big)=\Pi\Big(\frac{1}{2\pi}\int_{\mathbb R} \widehat{f}(t)e^{itb}\d t\Big)=\frac{1}{2\pi}\int_{\mathbb R} \widehat{f}(t)\Pi(e^{itb})\d t=\frac{1}{2\pi}\int_{\mathbb R} \widehat{f}(t)e^{it\Pi(b)}\d t,
\end{equation} 
and the right hand side corresponds to $f\big(\Pi(b)\big)$. Let $\A={\rm C^*}(1,\Pi(b))$ be the (commutative) $C^*$-algebra generated by $1$ and $\Pi(b)$. Now if $\chi$ is a character of $\A$, we see that
$$
\chi(\Pi\big(f(b)\big))\overset{\eqref{calculation}}{=}f(\chi\big(\Pi(b)\big))=\chi(f\big(\Pi(b)\big)),
$$ 
so $\chi(\Pi\big(f(b)\big))=\chi(f\big(\Pi(b)\big))$ holds for all the characters $\chi$ of $\A$ and hence $\Pi\big(f(b)\big)=f\big(\Pi(b)\big)$.

\smallskip
\emph{(iii)} Let $\Delta$ be the spectrum of $\B(b,1)$. By \emph{(ii)}, we have $${\rm Spec}_{\B(b,1)}\big(f(b)\big)=\{h\big(f(b)\big) \mid h\in\Delta\}=\{f\big(h(b)\big) \mid h\in\Delta\}=f\big({\rm Spec}_{\B(b,1)}(b)\big).$$

\emph{(iv)} Let $\Pi:\B(b,1)\to\BofH$ be a faithful non-degenerate $^*$-representation. For any $f,g\in C_{\rm c}^{d+2}(\mathbb R)$, we have 
$$
\Pi\big((f\cdot g)(b)\big)=(f\cdot g)\big(\Pi(b)\big)=f\big(\Pi(b)\big)g\big(\Pi(b)\big)=\Pi\big(f(b)\big)\Pi\big(g(b)\big)=\Pi\big(f(b)g(b)\big)
$$ 
and 
$$
\Pi\big(\overline{f}(b)\big)=\overline{f}\big(\Pi(b)\big)=f\big(\Pi(b)\big)^*\overset{\eqref{calculation}}{=}\Pi\big(f(b)\big)^*=\Pi\big(f(b)^*\big).
$$ 
So $\varphi_b$ is a $^*$-homomorphism.

\emph{(v)} We note that 
$$
(1+|t|)^{d+2}|\widehat{f}(t)|\leq 2R(d+2)!\sum_{k=0}^{d+2}\norm{f^{(k)}}_{C_{0}(\mathbb R)},
$$
holds for all $t\in\mathbb R$. The result then follows by repeating the same reasoning as in part \emph{(i)} with this more concrete bound. \end{proof}

\begin{rem}
    Suppose $\B$ is a unital Banach $^*$-algebra and $\varphi:C_{\rm c}^{d+2}(\mathbb R)\to \B(b,1)$ is a $^*$-homomorphism such that $\varphi(f)=1$ for some $f\in C_{\rm c}^{d+2}(\mathbb R)$. If $K={\rm Supp}(f)$, then $\varphi$ induces an $^*$-homomorphism $\varphi_0: C^{d+2}(K)\to \B$. 
\end{rem}

\section{Integrable cross-sections of polynomial growth}\label{mainsec}

From now on $\G$ will be a (Hausdorff) unimodular, locally compact group with unit $\e$ and Haar measure $d\mu(x)\equiv dx$. If $\G$ is compact, we assume that $\mu$ is normalized so that $\mu(\G)=1$. During most of the article, $\G$ will be assumed of \emph{polynomial growth of order $d$}, meaning that $$
\mu(K^n)=O(n^d),\quad\textup{ as }n\to\infty,
$$ for all relatively compact subsets $K\subset\G$. We will also fix a Fell bundle $\mathscr C\!=\bigsqcup_{x\in\G}\mathfrak C_x$ over $\G$. The algebra of integrable cross-sections $L^1(\G\,\vert\,\mathscr C)$ is a Banach $^*$-algebra and a completion of the space $C_{\rm c}(\G\,\vert\,\mathscr C)$ of continuous cross-sections with compact support. Its universal $C^*$-algebra completion is denoted by ${\rm C^*}(\G\,\vert\,\mathscr C)$. For the general theory of Fell bundles we followed \cite[Chapter VIII]{FD88}, to which we refer for details. We will only recall the product on $L^1(\G\,\vert\,\mathscr C)$, given by
\begin{equation}\label{broduct}
\big(\Phi*\Psi\big)(x)=\int_\G \Phi(y)\bu \Psi(y^{-1}x)\,\d y
\end{equation}
and its involution
\begin{equation}\label{inwol}
\Phi^*(x)=\Phi(x^{-1})^\bu\,,
\end{equation}
in terms of the operations $\big(\bu,^\bu\big)$ on the Fell bundle. We will also consider the $L^p$-spaces $L^p(\G\,\vert\,\mathscr C)$, endowed with the norms \begin{equation}
    \norm{\Phi}_{L^p(\G\,\vert\,\mathscr C)}=\left\{\begin{array}{ll}
\,\big(\int_\G \norm{\Phi(x)}_{\mathfrak C_x}^p \d x\big)^{1/p}    & \textup{if\ } p\in[1,\infty), \\
\,{\rm essup}_{x\in \G}\norm{\Phi(x)}_{\mathfrak C_x}     & \textup{if\ } p=\infty. \\
\end{array}\right.
\end{equation} If $\G$ is compact, $L^p(\G\,\vert\,\mathscr C)$ is a dense Banach $^*$-subalgebra of $L^1(\G\,\vert\,\mathscr C)$. In particular $C(\G\,\vert\,\mathscr C)$ is also a Banach $^*$-algebra when endowed with the $L^\infty$-norm. In this case, it is immediate that $p\leq q$ implies $L^q(\G\,\vert\,\mathscr C)\subset L^p(\G\,\vert\,\mathscr C)$. On the other hand, if $\G$ is discrete, we will write $\ell^p(\G\,\vert\,\mathscr C)$ instead of $L^p(\G\,\vert\,\mathscr C)$.

\begin{ex}\label{mainex}
    Let $\A$ be a $C^*$-algebra. A (continuous) twisted action of $\G$ on $\A$ is a pair $(\alpha,\omega)$ of continuous maps $\alpha:\G\to{\rm Aut}({\A})$, $\omega:\G\times\G\to \mathcal{UM}({\A})$, such that \begin{itemize}
        \item[(i)] $\alpha_x(\omega(y,z))\omega(x,yz)=\omega(x,y)\omega(xy,z)$,
        \item[(ii)] $\alpha_x\big(\alpha_y(a)\big)\omega(x,y)=\omega(x,y)\alpha_{xy}(a)$,
        \item[(iii)] $\omega(x,\e)=\omega(\e,y)=1, \alpha_\e={\rm id}_{{\A}}$,
    \end{itemize} for all $x,y,z\in\G$ and $a\in\A$. 

The quadruple $(\G,\A,\alpha,\omega)$ is called a \emph{twisted $C^*$-dynamical system}. Given such a twisted action, one usually forms the so-called \emph{twisted convolution algebra} $L^1_{\alpha,\omega}(\G,\A)$, consisting of all Bochner integrable functions $\Phi:\G\to\A$, endowed with the product 
\begin{equation}\label{convolution}
    \Phi*\Psi(x)=\int_\G \Phi(y)\alpha_y[\Psi(y^{-1}x)]\omega(y,y^{-1}x)\d y
\end{equation} and the involution \begin{equation}\label{involution}
    \Phi^*(x)=\omega(x,x^{-1})^*\alpha_x[\Phi(x^{-1})^*],
\end{equation} making $L^1_{\alpha,\omega}(\G,\A)$ a Banach $^*$-algebra under the norm $\norm{\Phi}_{L^1_{\alpha,\omega}(\G,\A)}=\int_\G\norm{\Phi(x)}_{\A}\d x$. When the twist is trivial ($\omega\equiv1$), we omit any mention of it and call the resulting algebra $L^1_{\alpha}(\G,\A)$ as the \emph{convolution algebra} associated with the action $\alpha$. In this case, the triple $(\G,\A,\alpha)$ is called a (untwisted) \emph{$C^*$-dynamical system}.

It is well-known that these algebras correspond to our algebras $L^1(\G\,\vert\,\mathscr C)$, for very particular Fell bundles. In fact, these bundles may be easily described as $\mathscr C=\A\times\G$, with quotient map $q(a,x)=x$, constant norms $\norm{\cdot}_{\mathfrak C_x}=\norm{\cdot}_{\A}$, and operations \begin{equation*}
    (a,x)\bu(b,y)=(a\alpha_x(b)\omega(x,y),xy)\quad\textup{and}\quad (a,x)^\bu=(\alpha_{x^{-1}}(a^*)\omega(x^{-1},x),x^{-1}).
\end{equation*} Due to the main theorem in \cite{ExLa97}, our results will also apply in the case of \emph{measurable} twisted actions, provided that $\G$ is second countable.
\end{ex}

It is convenient to have a spatial representation of ${\rm C^*}(\G\,\vert\,\mathscr C)$, so we will now introduce the left regular representation. It is quite natural to look at the representation of $L^1(\G\,\vert\,\mathscr C)$ as bounded operators on $L^2(\G\,\vert\,\mathscr C)$. Moreover, the interplay between these spaces is quite harmonious, but it carries the issue that $L^2(\G\,\vert\,\mathscr C)$ is not a Hilbert $C^*$-module, so it seems unlikely that we can get a $^*$-representation in this way. This forces us to consider the space $L^2_\e(\G\,\vert\,\mathscr C)$, the completion of $L^2(\G\,\vert\,\mathscr C)$ under the norm 
$$
\norm{\Phi}_{L^2_\e(\G\,\vert\,\mathscr C)}=\norm{\int_\G\Phi(x)^\bu\bu\Phi(x)\,\d x}_{\mathfrak C_\e}^{1/2}.
$$ 
This is a Hilbert $C^*$-module over $\mathfrak C_\e$, so the set of adjointable operators is a $C^*$-algebra under the operator norm. We denote this algebra by $\mathbb B_a(L^2_\e(\G\,\vert\,\mathscr C))$. The left regular representation $\lambda$ is then the $^*$-monomorphism given by 
\begin{equation*}
    \lambda:L^1(\G\,\vert\,\mathscr C)\to \mathbb B_a(L^2_\e(\G\,\vert\,\mathscr C)), \textup{ defined by }\lambda(\Phi)\Psi=\Phi*\Psi, \textup{ for all }\Psi\in L^2(\G\,\vert\,\mathscr C).
\end{equation*} 
It follows from the results in \cite[Theorem 3.9]{ExNg02} that, for amenable groups, the universal $C^*$-completion (the so-called $C^*$-envelope) of $L^1(\G\,\vert\,\mathscr C)$, previously denoted ${\rm C^*}(\G\,\vert\,\mathscr C)$, coincides with the norm-closure of $\overline{\lambda(L^1(\G\,\vert\,\mathscr C))}$. 

Now we would like to do some spectral analysis, but in order to do so, we feel the need to define $L^p$ versions of the norm $\norm{\cdot}_{L^2_\e(\G\,\vert\,\mathscr C)}$. Our inspiration then comes from \cite[Definition 2.1]{LeNg06}. 

\begin{defn}\label{newnorms}
    Let $\pi:\mathscr C\to\BofH$ be a continuous non-degenerate $^*$-representation of $\mathscr C$ on the Hilbert space $\mathcal H$. One can use this representation to define the norms \begin{equation}
        \norm{\Phi}_{\pi,p}=\sup_{\norm{\xi}_{\mathcal H}\leq\,1} \Big(\int_\G\norm{\pi\big(\Phi(x)\big)\xi}_{\mathcal H}^p\,\d x\Big)^{1/p},
    \end{equation} for $p\in [1,\infty)$ and $\Phi\in C_{\rm c}(\G\,\vert\,\mathscr C)$. As $\norm{\Phi}_{\pi,p}\leq \norm{\Phi}_{L^p(\G\,\vert\,\mathscr C)}$, it is immediate that these norms are finite.
\end{defn}

\begin{rem}
    For $\pi:\mathscr C\to\BofH$, a continuous $^*$-representation of Fell bundles, preserving the norms of the fibers is the same as $\pi|_{\mathfrak C_\e}$ being a faithful $^*$-representation of $\mathfrak C_\e$. Indeed, for all $a\in\mathfrak C_x$, $$\norm{\pi(a)}_{\BofH}^2=\norm{\pi(a^\bu\bu a)}_{\BofH}=\norm{a^\bu\bu a}_{\mathfrak C_\e}=\norm{a}_{\mathfrak C_x}^2.$$ For this reason, we will say that these representations are \emph{isometric}.
\end{rem}

The next lemma compiles some facts about the norms $\norm{\cdot}_{\pi,p}$ and their relations with the other norms previously introduced. It can be regarded as both our new approach to the problem and our main computation tool.

\begin{lem}\label{lemma}
    Let $\pi:\mathscr C\to\BofH$ be an isometric, non-degenerate $^*$-representation and let $\Phi,\Psi\in C_{\rm c}(\G\,\vert\,\mathscr C)$. The following are true.
     \begin{enumerate} 
        \item[(i)] $\norm{\Phi}_{L^p(\G\,\vert\,\mathscr C)}=\Big(\int_\G\norm{\pi\big(\Phi(x)\big)}_{\BofH}^p\,\d x\Big)^{1/p}$, for all $p\in [1,\infty)$.
        \item[(ii)] $\norm{\Phi}_{\pi,2}=\norm{\Phi}_{L^2_\e(\G\,\vert\,\mathscr C)}$.
        \item[(iii)] $\lim_{p\to\infty} \norm{\Phi}_{\pi,p}=\norm{\Phi}_{L^\infty(\G\,\vert\,\mathscr C)}$.
        \item[(iv)] $\norm{\Psi*\Phi}_{\pi,r}\leq \norm{\Psi}_{L^p(\G\,\vert\,\mathscr C)}\norm{\Phi}_{\pi,q}$, for all $p,q,r\in[1,\infty)$ satisfying $1/p+1/q=1+1/r$.
        \item[(v)] If $p,q\in[1,\infty)$ satisfy $1/p+1/q=1$, then $\norm{\Psi*\Phi}_{L^\infty(\G\,\vert\,\mathscr C)}\leq \norm{\Psi}_{L^p(\G\,\vert\,\mathscr C)}\norm{\Phi}_{\pi,q}$.
        \item[(vi)] If $\G$ is discrete and $1\leq p\leq q<\infty$, then $\norm{\Phi}_{\ell^\infty(\G\,\vert\,\mathscr C)}\leq \norm{\Phi}_{\pi,q}\leq \norm{\Phi}_{\pi,p}$ and $\norm{\Phi}_{\ell^2_\e(\G\,\vert\,\mathscr C)}\leq \norm{\Phi}_{{\rm C^*}(\G\,\vert\,\mathscr C)}$. \end{enumerate} \end{lem}
\begin{proof}
    \emph{(i)} Due to the previous remark, it follows easily that $$\Big(\int_\G\norm{\pi\big(\Phi(x)\big)}_{\BofH}^p\,\d x\Big)^{1/p}=\Big(\int_\G\norm{\Phi(x)}_{\mathfrak C_x}^p\,\d x\Big)^{1/p}=\norm{\Phi}_{L^p(\G\,\vert\,\mathscr C)}.$$ 
    
    \emph{(ii)} We have $$\int_\G \,\norm{\pi(\Phi(x))\xi}_{\mathcal H}^2\,\d x=\Big\langle \int_\G \pi\big(\Phi(x)^\bu\bu\Phi(x)\big)\xi\,\d x,\xi\Big\rangle=\Big\langle \pi\big(\int_\G \Phi(x)^\bu\bu\Phi(x)\,\d x\big)\xi,\xi\Big\rangle,$$ implying that \begin{equation*}
        \norm{\Phi}_{\pi,2}^2=\norm{\pi\big(\int_\G \Phi(x)^\bu\bu\Phi(x)\,\d x\big)}_{\BofH}=\norm{\int_\G \Phi(x)^\bu\bu\Phi(x)\,\d x}_{\mathfrak C_\e}^2=\norm{\Phi}_{L^2_\e(\G\,\vert\,\mathscr C)}^2,
    \end{equation*} 
    where the second equality follows from the fact that $\pi$ is isometric.

    \emph{(iii)} We have \begin{align*}
        \norm{\Phi}_{\pi,p}\leq\sup_{\norm{\xi}_{\mathcal H}\leq\,1}\Big(\int_\G \norm{\Phi(x)}_{\mathfrak C_x}^{p-q}\norm{\pi\big(\Phi(x)\big)\xi}_{\mathcal H}^{q} \,\d x\Big)^{1/p}\leq \norm{\Phi}_{L^\infty(\G\,\vert\,\mathscr C)}^{\frac{p-q}{p}} \norm{\Phi}_{\pi,p}^{\frac{q}{p}}
    \end{align*} and hence $\limsup_{p\to\infty} \norm{\Phi}_{\pi,p}\leq \norm{\Phi}_{L^\infty(\G\,\vert\,\mathscr C)}$. On the other hand, for $\xi\in \mathcal H$, with $\norm{\xi}_{\mathcal H}\leq\,1$, and $0<\delta<\norm{\Phi}_{L^\infty(\G\,\vert\,\mathscr C)}$, we consider the set $$D_\delta=\{x\in\G\mid \norm{\pi\big(\Phi(x)\big)\xi}_{\mathcal H}\geq\norm{\norm{\pi\big(\Phi(\cdot)\big)\xi}_{\mathcal H}}_{L^\infty(\G)}-\delta\}.$$ We then have \begin{align*}
        \norm{\Phi}_{\pi,p}&\geq \sup_{\norm{\xi}_{\mathcal H}\leq\,1}\Big(\int_{D_\delta} (\norm{\norm{\pi\big(\Phi(\cdot)\big)\xi}_{\mathcal H}}_{L^\infty(\G)}-\delta)^p \,\d x\Big)^{1/p} \\
        &=\sup_{\norm{\xi}_{\mathcal H}\leq\,1}(\norm{\norm{\pi\big(\Phi(\cdot)\big)\xi}_{\mathcal H}}_{L^\infty(\G)}-\delta)\mu(D_\delta)^{1/p} \\
        &=(\norm{\Phi}_{L^\infty(\G\,\vert\,\mathscr C)}-\delta)\mu(D_\delta)^{1/p}
    \end{align*} so $\liminf_{p\to\infty}\norm{\Phi}_{\pi,p}\geq \norm{\Phi}_{L^\infty(\G\,\vert\,\mathscr C)}$.

    \emph{(iv)} We let $p_1=\frac{pr}{r-p}$ and $q_1=\frac{qr}{r-q}$, so that
    $$
    \frac{(r-p)p_1}{r}=p,\quad \frac{(r-q)q_1}{r}=q \quad\text{and}\quad \frac{1}{p_1}+\frac{1}{q_1}+\frac{1}{r}=1.
    $$
    The last equality allows us to use H\"older's inequality with the exponents $p_1,q_1,r$. Indeed, for $x\in\G$, we have 
    \begin{align*}
        \norm{\pi\big(\Psi*\Phi(x)\big)\xi}_{\mathcal{H}}&=\norm{\int_\G \pi\big(\Psi(y)\big)\pi\big(\Phi(y^{-1}x)\big)\xi\,\d y}_{\mathcal H} \\
        &\leq \int_\G \norm{\Psi(y)}_{\mathfrak C_y}\norm{\pi\big(\Phi(y^{-1}x)\big)\xi}_{\mathcal H} \,\d y \\
        &= \int_\G \Big(\norm{\Psi(y)}_{\mathfrak C_y}^p\norm{\pi\big(\Phi(y^{-1}x)\big)\xi}_{\mathcal H}^q\Big)^{1/r}\norm{\Psi(y)}_{\mathfrak C_y}^{(r-p)/r}\norm{\pi\big(\Phi(y^{-1}x)\big)\xi}_{\mathcal H}^{(r-q)/r} \,\d y       \\
        &\leq \norm{\Psi}_{L^p(\G\,\vert\,\mathscr C)}^{p/p_1}\norm{\Phi}_{\pi,q}^{q/q_1}\Big(\int_\G \norm{\Psi(y)}_{\mathfrak C_y}^p\norm{\pi\big(\Phi(y^{-1}x)\big)\xi}_{\mathcal H}^q \,\d y\Big)^{1/r}.
    \end{align*}
    And integrating on $x$ yields 
    \begin{align*}
        \int_\G\norm{\pi\big(\Psi*\Phi(x)\big)\xi}_{\mathcal{H}}^r\,\d x&\leq \norm{\Psi}_{L^p(\G\,\vert\,\mathscr C)}^{rp/p_1}\norm{\Phi}_{\pi,q}^{rq/q_1} \int_\G\int_\G \norm{\Psi(y)}_{\mathfrak C_y}^p\norm{\pi\big(\Phi(y^{-1}x)\big)\xi}_{\mathcal H}^q \,\d y\d x \\
        &=\norm{\Psi}_{L^p(\G\,\vert\,\mathscr C)}^{p+rp/p_1}\norm{\Phi}_{\pi,q}^{q+rq/q_1} \\
        &=\norm{\Psi}_{L^p(\G\,\vert\,\mathscr C)}^{r}\norm{\Phi}_{\pi,q}^{r},
    \end{align*} 
    which proves the claim. 
    
    \emph{(v)} We proceed as in the beginning of \emph{(iv)}, but using $p_1=p$ and $q_1=q$, to get the inequality $$\norm{\pi\big(\Psi*\Phi(x)\big)\xi}_{\mathcal{H}}\leq \norm{\Psi}_{L^p(\G\,\vert\,\mathscr C)}\norm{\Phi}_{\pi,q}.$$ Then, $$\norm{\Psi*\Phi(x)}_{\mathfrak C_x}=\sup_{\norm{\xi}_{\mathcal H}\leq\,1} \norm{\pi\big(\Psi*\Phi(x)\big)\xi}_{\mathcal{H}}\leq \norm{\Psi}_{L^p(\G\,\vert\,\mathscr C)}\norm{\Phi}_{\pi,q}$$ and therefore $\norm{\Psi*\Phi}_{L^\infty(\G\,\vert\,\mathscr C)}\leq \norm{\Psi}_{L^p(\G\,\vert\,\mathscr C)}\norm{\Phi}_{\pi,q}$. 
    
    \emph{(vi)} For any $x\in\G$, we have 
        \begin{align*}
         \norm{\Phi(x)}_{\mathfrak C_x}=\sup_{\norm{\xi}_{\mathcal H}\leq\,1}\norm{\pi\big(\Phi(x)\big)\xi}_{\mathcal H}\leq\sup_{\norm{\xi}_{\mathcal H}\leq\,1} \Big(\sum_{x\in\G}\norm{\pi\big(\Phi(x)\big)\xi}_{\mathcal H}^p\Big)^{1/p}=\norm{\Phi}_{\pi,p}
        \end{align*} 
    hence 
    $$
    \norm{\Psi}_{\ell^\infty(\G\,\vert\,\mathscr C)}\leq\norm{\Phi}_{\pi,p}.
    $$ 
    On the other hand, if $\{a_\alpha\}\subset \mathfrak C_\e$ is an approximate unit and let $\Phi\bu a_\alpha$ be the section defined by $\Phi\bu a_\alpha(x)=\Phi(x)\bu a_\alpha$, then
    $$
    \norm{\Phi}_{\ell^2_\e(\G\,\vert\,\mathscr C)}=\lim_\alpha\norm{\Phi \bu a_\alpha}_{\ell^2_\e(\G\,\vert\,\mathscr C)}=\lim_\alpha\norm{\Phi*a_\alpha\delta_\e}_{\ell^2_\e(\G\,\vert\,\mathscr C)}\leq \norm{\Phi}_{{\rm C^*}(\G\,\vert\,\mathscr C)},
    $$ 
    finishing the proof.\end{proof}

Lemma \ref{lemma} is fundamental because of the following reasons: the $L^2_\e$-norm is recovered by $\norm{\cdot}_{\pi,2}$, so we can link the norms $\norm{\cdot}_{\pi,q}$ to the $L^p$-norms via the strengthened Young-type inequality in \ref{lemma}\emph{(v)}. However, this connection is still rather weak (at least compared to the usual relations the $L^p$-norms satisfy) and it will force us to find indirect ways to do our proofs. This is illustrated by the proofs of Proposition \ref{bigO},  Lemma \ref{exist}, Theorem \ref{closedness}, Proposition \ref{subexp}, among others.

We will now focus our efforts in constructing the functional calculus. In order to do so, let us consider the following entire functions $u,v:\mathbb C\to \mathbb C$, given by \begin{equation}\label{functions}
    u(z)=e^{iz}-1=\sum_{k=1}^{\infty} \frac{i^k z^k}{k!},\quad v(z)=\frac{e^{iz}-1-iz}{z}=\sum_{k=0}^{\infty} \frac{-i^{k} z^{k+1}}{(k+2)!}.
\end{equation} It is clear that $u(z)=v(z)z+iz$. 

\begin{rem}\label{nonunital}
    It is also clear that an element $b\in\B$ has polynomial growth of order $d$ if and only if $\norm{u(tb)}_{\B}=O(|t|^d)$, as $|t|\to\infty$.
\end{rem}

\begin{lem}\label{Dix}
    Let $\Phi=\Phi^*\in L^1(\G\,\vert\,\mathscr C)\cap L^2(\G\,\vert\,\mathscr C)$. Then $v(\Phi)\in  L^2(\G\,\vert\,\mathscr C)$ and \begin{equation}\label{dixlem}
        \norm{v(\Phi)}_{ L^2_\e(\G\,\vert\,\mathscr C)}\leq \tfrac{1}{2}\norm{\Phi}_{ L^2_\e(\G\,\vert\,\mathscr C)}.
    \end{equation}\end{lem}

\begin{proof}
    $v(\Phi)$ belongs to both $L^1(\G\,\vert\,\mathscr C)$ and $L^2(\G\,\vert\,\mathscr C)\subset L^2_\e(\G\,\vert\,\mathscr C)$, as \begin{equation}\label{convergence}
        \norm{v(\Phi)}_{L^2(\G\,\vert\,\mathscr C)}\leq \sum_{k=0}^{\infty} \frac{1}{(k+2)!}\norm{\Phi^{k+1}}_{L^2(\G\,\vert\,\mathscr C)}\leq \norm{\Phi}_{L^2(\G\,\vert\,\mathscr C)} \sum_{k=0}^{\infty} \frac{1}{(k+2)!}\norm{\Phi}_{L^1(\G\,\vert\,\mathscr C)}^{k}<\infty.
    \end{equation} Now consider the entire (complex) function $$w(z)=\frac{v(z)}{z}=\sum_{k=0}^{\infty} \frac{-i^{k} z^{k}}{(k+2)!}.$$ It satisfies $w(z)z=v(z)$ and therefore $w\big(\lambda(\Phi)\big)\Phi=v(\Phi)$. So, if ${\rm Spec}(a)$ denotes the spectrum of an element $a$ in $\mathbb B_a(L^2_\e(\G\,\vert\,\mathscr C))$, we have \begin{align*}
        \norm{v(\Phi)}_{L^2_\e(\G\,\vert\,\mathscr C)}&\leq\norm{w\big(\lambda(\Phi)\big)}_{\mathbb B(L^2_\e(\G\,\vert\,\mathscr C))}\norm{\Phi}_{L^2_\e(\G\,\vert\,\mathscr C)} \\
        &=\sup_{\alpha\in {\rm Spec}(\lambda(\Phi))} |w(\alpha)|\,\norm{\Phi}_{L^2_\e(\G\,\vert\,\mathscr C)}\\
        &\leq\sup_{\alpha\in \mathbb R}|w(\alpha)|\,\norm{\Phi}_{L^2_\e(\G\,\vert\,\mathscr C)} \\
        &\leq \tfrac{1}{2}\norm{\Phi}_{L^2_\e(\G\,\vert\,\mathscr C)},
    \end{align*} finishing the proof. \end{proof}


\begin{prop}\label{bigO}
    Let $\G$ be a group of polynomial growth of order $d$ and let $\Phi=\Phi^*\in C_{\rm c}(\G\,\vert\,\mathscr C)$. Then \begin{equation}
        \norm{u(t\Phi)}_{L^1(\G\,\vert\,\mathscr C)}=O(|t|^{2d+2}),\quad\textup{ as }|t|\to\infty.
    \end{equation}\end{prop}

\begin{proof}
    Let $K$ be a compact subset of $\G$, containing both ${\rm Supp}(\Phi)$ and $\e$; this implies that $K\subset K^2\subset\ldots\subset K^n$. Let $n$ be a positive integer. Then we observe that 
    \begin{equation}\label{decomp}
        \norm{u(n\Phi)}_{L^1(\G\,\vert\,\mathscr C)}=\int_{K^{n^2-1}}  \norm{u(n\Phi)(x)}_{\mathfrak C_x}\d x+\int_{\G \setminus K^{n^2-1}} \norm{u(n\Phi)(x)}_{\mathfrak C_x}\d x.
    \end{equation}
    The second integral is easier to bound. Indeed, note that $\Phi, \Phi^2,\ldots, \Phi^{n^2-1}$ all vanish in $\G\setminus K^{n^2-1}$. This means that
    \begin{align*}
        \int_{\G \setminus K^{n^2-1}} \norm{u(n\Phi)(x)}_{\mathfrak C_x}\d x &\leq \int_{\G\setminus K^{n^2-1}} \sum_{k=n^2}^{\infty} \frac{n^k}{k!}\norm{\Phi^{k}(x)}_{\mathfrak C_x}\d x \\
        &\leq \sum_{k=n^2}^{\infty} \frac{n^k}{k!}\norm{\Phi}_{L^1(\G\,\vert\,\mathscr C)}^k \\
        &\leq \frac{n^{n^2}}{(n^2)!}e^{n\norm{\Phi}_{L^1(\G\,\vert\,\mathscr C)}}\norm{\Phi}_{L^1(\G\,\vert\,\mathscr C)}^{n^2}.
    \end{align*} 
    The right hand side is a bounded sequence. On the other hand, the first integral in \eqref{decomp} requires a more detailed analysis. We now note that 
    $$
    u(n\Phi)=\sum_{k=1}^{\infty} \frac{i^k n^k\Phi^k}{k!}=in\Phi+n\Phi*\big(\sum_{k=2}^{\infty} \frac{i^k n^{k-1}\Phi^{k-1}}{k!}\big)=in\Phi+n\Phi*v(n\Phi).
    $$
    So we use Lemma \ref{lemma} to get
    \begin{align*}
        \norm{u(n\Phi)(x)}_{\mathfrak C_x}&\leq  n\norm{\Phi(x)}_{\mathfrak C_x}+n\norm{\Phi*v(n\Phi)}_{L^\infty(\G\,\vert\,\mathscr C)}\\
        &\leq n\norm{\Phi(x)}_{\mathfrak C_x}+n\norm{\Phi}_{L^2(\G\,\vert\,\mathscr C)}\norm{v(n\Phi)}_{L^2_\e(\G\,\vert\,\mathscr C)},
    \end{align*}
    for all $x\in\G$. Now, we proceed by H\"older's inequality  
    \begin{align*}
    \int_{K^{n^2-1}}  \norm{u(n\Phi)(x)}_{\mathfrak C_x}\d x &\leq n\Big(\norm{\Phi}_{L^1(\G\,\vert\,\mathscr C)}+\mu(K^{n^2-1})\norm{\Phi}_{L^2(\G\,\vert\,\mathscr C)}\norm{v(n\Phi)}_{L^2_\e(\G\,\vert\,\mathscr C)}\Big) \\
    &\overset{\eqref{dixlem}}{\leq} n\Big(\norm{\Phi}_{L^1(\G\,\vert\,\mathscr C)}+\tfrac{n}{2}\mu(K^{n^2-1})\norm{\Phi}_{L^2(\G\,\vert\,\mathscr C)}^2\Big).
    \end{align*} 
    But, by assumption, $\mu(K^{n^2-1})=O\big((n^2-1)^{d}\big)=O( n^{2d}),$ as $n\to\infty$. Hence we have shown that  $\norm{u(n\Phi)}_{L^1(\G\,\vert\,\mathscr C)}=O(n^{2d+2})$, as $n\to\infty$. Now, for $t\in\mathbb R$, we take its integer part $n=\lfloor t\rfloor$ and see that \begin{align*}
        \norm{e^{it\Phi}}_{\widetilde{L^1(\G\,\vert\,\mathscr C)}}&\leq \norm{e^{in\Phi}}_{\widetilde{L^1(\G\,\vert\,\mathscr C)}}\norm{e^{i(t-n)\Phi}}_{\widetilde{L^1(\G\,\vert\,\mathscr C)}} \\
        &\leq (1+\norm{u(\lfloor t\rfloor\Phi)}_{L^1(\G\,\vert\,\mathscr C)})e^{\norm{\Phi}_{L^1(\G\,\vert\,\mathscr C)}} \\
        &=O(|t|^{2d+2}),
    \end{align*} which proves the result.
\end{proof}

\begin{rem}
    If we compare Proposition \ref{bigO} to Dixmier's result \cite[Lemme 6]{Di60}, it is obvious that, when restricted to his setting, Dixmier's result is better, as it provides a slower growth and therefore a bigger functional calculus. This happens, of course, due to the simpler setting, but it makes us wonder if a slower growth is possible in general.
\end{rem}

\begin{rem}
    The above proposition shows that, for a fixed $\Phi\in C_{\rm c}(\G\,\vert\,\mathscr C)_{\rm sa}$ and a large enough $|t|$, there is a constant $A>0$ so that
    \begin{equation}\label{multiple}
        \norm{u(t\Phi)}_{L^1(\G\,\vert\,\mathscr C)}\leq A|t|^{2d+2}.
    \end{equation}
    We remark that the proof of Proposition \ref{bigO} reveals that the constant $A$ only depends on the support of $\Phi$, and the quantities $\norm{\Phi}_{L^1(\G\,\vert\,\mathscr C)},\norm{\Phi}_{L^2(\G\,\vert\,\mathscr C)}$. It therefore follows, from the same reasoning, that the growth conclusion \eqref{multiple} can be obtained for a family $\Phi\in F\subset C_{\rm c}(\G\,\vert\,\mathscr C)_{\rm sa}$, provided that 
    \begin{equation*}
    \sup_{\Phi\in F} \norm{\Phi}_{L^1(\G\,\vert\,\mathscr C)}<\infty,\quad  \sup_{\Phi\in F} \norm{\Phi}_{L^2(\G\,\vert\,\mathscr C)}<\infty\quad \text{and}\quad \bigcup_{\Phi\in F}{\rm Supp}(\Phi)\text{ is relatively compact.}
    \end{equation*}
\end{rem}

Thanks to the Dixmier-Baillet theorem, we have the following functional calculus for all the elements in $C_{\rm c}(\G\,\vert\,\mathscr C)_{\rm sa}$.

\begin{thm}\label{smoothfuncalc}
     Let $\G$ be a group of polynomial growth of order $d$ and $\Phi=\Phi^*\in C_{\rm c}(\G\,\vert\,\mathscr C)$. Let $f:\mathbb R\to \mathbb C$ be a complex function that admits $2d+4$ continuous and integrable derivatives, with Fourier transform $\widehat f$. Then \begin{enumerate}
         \item[(i)] $f(\Phi)=\frac{1}{2\pi}\int_{\mathbb R} \widehat{f}(t)e^{it\Phi}\d t$ exists in $\widetilde{L^1(\G\,\vert\,\mathscr C)}$, or in $L^1(\G\,\vert\,\mathscr C)$ if $f(0)=0$.
         \item[(ii)] For any non-degenerate $^*$-representation $\Pi:L^1(\G\,\vert\,\mathscr C)\to\BofH$, we have $$\widetilde{\Pi}(f\big(\Phi\big))=f\big(\Pi(\Phi)\big).$$ Furthermore, if $\Pi$ is injective, we also have $${\rm Spec}_{\widetilde{L^1(\G\,\vert\,\mathscr C)}}\big(f(\Phi)\big)={\rm Spec}_{\BofH}\big(\widetilde\Pi\big(f(\Phi)\big)\big).$$ 
     \end{enumerate}
\end{thm}
\begin{proof}
    Combine Proposition \ref{bigO} and Theorem \ref{DixBai} with Remark \ref{nonunital}.
\end{proof}

\section{A dense symmetric subalgebra}\label{symmsub}

Besides the functional calculus we defined and in order to get better spectral properties, we would like for our algebras to be symmetric. The relevant definition is

\begin{defn} \label{symmetric}
A Banach $^*$-algebra $\mathfrak B$ is called {\it symmetric} if the spectrum of $b^*b$ is positive for every $b\in\mathfrak B$.
\end{defn}

\begin{rem}
    A Banach $^*$-algebra $\mathfrak B$ is symmetric if and only if the spectrum of any self-adjoint element is real. This is, in fact, the content of the celebrated Shirali-Ford theorem.
\end{rem}

\begin{rem}\label{symm}
The symmetry of $L^{1}(\G\,\vert\,\mathscr C)$ itself seems to be a very complicated business, despite the fact that many of its self-adjoint elements have real spectrum (this is a consequence of Proposition \ref{bigO}, but it is also proven directly in Proposition \ref{subexp}). In general, $L^{1}(\G\,\vert\,\mathscr C)$ is known to be symmetric when $\G$ is nilpotent \cite{FJM} or compact (Theorem \ref{corcompact}). Some particular examples include $L^1_{\rm lt}(\G,C_0(\G))$, for any group $\G$ \cite[Theorem 4]{LP79}, or the examples constructed in \cite[Theorem 16]{Ku79}, among others.
\end{rem}

In any case, we will construct a dense $^*$-subalgebra of $L^{1}(\G\,\vert\,\mathscr C)$ that will be symmetric under mild conditions (such as compact generation of $\G$). This algebra is constructed using weights on the group, so we introduce the notion of weight now.

\begin{defn}
    A {\it weight} on the locally compact group $\G$ is a measurable, locally bounded function $\nu: \G\to [1,\infty)$ satisfying 
\begin{equation*}\label{submultiplicative}
 \nu(xy)\leq \nu(x)\nu(y)\,,\quad\nu(x^{-1})=\nu(x)\,,\quad\forall\,x,y\in\G\,.
\end{equation*}
\end{defn} This gives rise to the Banach $^*$-algebra $L^{1,\nu}(\G)$, defined as the completion of $C_{\rm c}(\G)$ under the norm
\begin{equation*}\label{ponderata}
\p\!\psi\!\p_{L^{1,\nu}(\G)}\,:=\int_\G\nu(x)|\psi(x)|\d x\,.
\end{equation*} This algebra has been studied, for example, in \cite{DLM04,FGLL03,Py82,SaSh19}, to which we refer for examples of weights. The Fell bundle analog is immediate. Let $\mathscr C$ be a Fell bundle over $\G$. On $\,C_{\rm c}(\G\,\vert\,\mathscr C)$  we introduce the norms
\begin{equation}\label{normix}
    \norm{\Phi}_{L^{p,\nu}(\G\,\vert\,\mathscr C)}=\left\{\begin{array}{ll}
\,\big(\int_\G \nu(x)^p\norm{\Phi(x)}_{\mathfrak C_x}^p \d x\big)^{1/p}    & \textup{if\ } p\in[1,\infty), \\
\,{\rm essup}_{x\in \G}\nu(x)\norm{\Phi(x)}_{\mathfrak C_x}     & \textup{if\ } p=\infty. \\
\end{array}\right.
\end{equation}
The completion in this norm is denoted by $L^{p,\nu}(\G\,\vert\,\mathscr C)$ and in the case $p=1$, it becomes a Banach $^*$-algebra with the $^*$-algebraic structure inherited from $L^1(\G\,\vert\,\mathscr C)$.

\begin{defn}
    A weight $\nu$ is called \emph{polynomial} if there is a constant $C>0$ such that \begin{equation}\label{polyweight}
        \nu(xy)\leq C\big(\nu(x)+\nu(y)\big),
    \end{equation} for all $x,y\in \G$.
\end{defn}

If $\G$ is compactly generated, we can always construct such a weight. In fact, such groups are characterized by possessing polynomial weights of a certain growth.
\begin{ex}\label{construction}
    Suppose that $\G$ is compactly generated and let $K$ be a relatively compact set satisfying $\G=\bigcup_{n\in\N}K^n$ and $K=K^{-1}$. Define $\si_K(x)=\inf\{n\in \N\mid x\in K^n\}$. Then $\si_K(xy)\leq \si_K(x)+\si_K(y)$ and therefore $$\nu_K(x)=1+\si_K(x)$$ defines a polynomial weight. 
\end{ex}

\begin{prop}\label{pytlik}
    Let $\G$ be a locally compact, compactly generated group and let $K\subset \G$ be as in Example \ref{construction}. The following are true \begin{enumerate}
        \item[(i)] Let $\nu$ be a polynomial weight, then there exist positive constants $M,\delta$ such that $$\nu(x)\leq M\nu_K(x)^\delta,$$ for all $x\in\G$.
        \item[(ii)] There exists a polynomial weight $\nu$ such that $x\mapsto\frac{1}{\nu(x)}$ belongs to $L^1(\G)$ (or any $L^p(\G)$, with $1\leq p<\infty$) if and only if $\G$ has polynomial growth. In such a case and if the growth is of order $d$, then $\nu_K^{-d-2}\in L^1(\G)$.
    \end{enumerate}
\end{prop}

\begin{proof}
    See \cite[Proposition 1, Proposition 2]{Py82}.
\end{proof}

If the group of study is not compactly generated, it seems hard to construct polynomial weights so that their inverse is $p$-integrable. The only other case that we can handle with full generality is made precise in the next example. It can be found in \cite[Example 1]{Py82}.

\begin{ex}\label{locallyfinite}
    Suppose there is an increasing sequence $\{\G_n\}_{n\in\N}$ of closed subgroups of $\G$ such that $\G=\bigcup_{n\in\N}\G_n$. Then for any increasing non-negative sequence $\{m_n\}_{n\in\N}$, we can form the function $\nu:\G\to[1,\infty)$ given by $$\nu=\sum_{n\in\N}m_n\chi_{\G_{n+1}\setminus \G_n},$$ where $\chi_A$ is the indicator function associated with the set $A$. It even satisfies $$\nu(xy)\leq\max\{\nu(x),\nu(y)\},$$ for all $x,y\in\G$. Thus $\nu$ is a weight if and only if it is locally bounded, which happens in the following cases: \begin{enumerate}
        \item[\emph{(i)}] $\G_n$ is compact for all $n\in\N$. In particular, if $\G$ is countable and locally finite, numerated as $\G=\{g_n\}_{n\in\N}$, one could use the finite subgroups $\G_n=\langle g_1,\ldots,g_n\rangle$.
        \item[\emph{(ii)}] If any compact subset $K\subset\G$ is fully contained in some $\G_n$.
    \end{enumerate} By adjusting the sequence $\{m_n\}_{n\in\N}$, we can easily force $\nu^{-1}\in L^p(\G)$, for any $0<p<\infty$.
\end{ex}

Before proceeding, let us introduce some concepts useful to the theory of symmetric Banach $^*$-algebras.
\begin{defn}\label{spinv}
    Let $\A\subset\B$ be a continuous inclusion of Banach $^*$-algebras. We say that: 
    \begin{itemize}
        \item[\emph{(i)}] $\A$ is {\it inverse-closed} or {\it spectrally invariant} in $\B$ if ${\rm Spec}_\A(a)={\rm Spec}_\B(a)$, for all $a\in\A$.
        \item[\emph{(ii)}] $\A$ is {\it spectral radius preserving} in $\B$ if $\rho_\A(a)=\rho_\B(a)$, for all $a\in\A$.
    \end{itemize}
\end{defn}

The following lemma is due to Barnes \cite[Proposition 2]{Ba00}. 
\begin{lem}\label{barnes}
    Let $\A\subset\B$ be a spectral radius preserving, continuous, dense inclusion of Banach $^*$-algebras. Then $\A$ is inverse-closed in $\B$.
\end{lem}

\begin{prop}\label{L1Lnu}
    Let $\nu$ be a polynomial weight on $\G$. Then $L^{1,\nu}(\G\,\vert\,\mathscr C)$ is inverse-closed in $L^{1}(\G\,\vert\,\mathscr C)$.
\end{prop}
\begin{proof}
    We note that for $\Phi,\Psi\in L^{1,\nu}(\G\,\vert\,\mathscr C) $, \begin{equation}
        \norm{\Psi*\Phi}_{L^{1,\nu}(\G\,\vert\,\mathscr C)}\leq C(\norm{\Psi}_{L^{1,\nu}(\G\,\vert\,\mathscr C)}\norm{\Phi}_{L^{1}(\G\,\vert\,\mathscr C)}+\norm{\Psi}_{L^{1}(\G\,\vert\,\mathscr C)}\norm{\Phi}_{L^{1,\nu}(\G\,\vert\,\mathscr C)}).
    \end{equation} Indeed, \begin{align*}
        \norm{\Psi*\Phi}_{L^{1,\nu}(\G\,\vert\,\mathscr C)}&\leq \int_\G\int_\G \norm{\Psi(y)\bu\Phi(y^{-1}x)}_{\mathfrak C_x} \nu(x)\d x\d y \\
        &\leq C\int_\G\int_\G\nu(y)\norm{\Psi(y)}_{\mathfrak C_{y}}\norm{\Phi(y^{-1}x)}_{\mathfrak C_{y^{-1}x}} +\nu(y^{-1}x)\norm{\Psi(y)}_{\mathfrak C_{y}}\norm{\Phi(y^{-1}x)}_{\mathfrak C_{y^{-1}x}} \d x\d y \\
        &=C(\norm{\Psi}_{L^{1,\nu}(\G\,\vert\,\mathscr C)}\norm{\Phi}_{L^{1}(\G\,\vert\,\mathscr C)}+\norm{\Psi}_{L^{1}(\G\,\vert\,\mathscr C)}\norm{\Phi}_{L^{1,\nu}(\G\,\vert\,\mathscr C)}).
    \end{align*} It then follows that \begin{equation}\label{diffeq}
        \norm{\Psi^{2n}}_{L^{1,\nu}(\G\,\vert\,\mathscr C)}\leq 2C\norm{\Psi^n}_{L^{1,\nu}(\G\,\vert\,\mathscr C)}\norm{\Psi^n}_{L^{1}(\G\,\vert\,\mathscr C)}
    \end{equation} so \begin{align*}
        \rho_{L^{1,\nu}(\G\,\vert\,\mathscr C)}(\Psi)^2&=\lim_{n\to\infty} \norm{\Psi^{2n}}^{1/n}_{L^{1,\nu}(\G\,\vert\,\mathscr C)} \\
        &\leq \lim_{n\to\infty}(2C)^{1/n}\norm{\Psi}_{L^{1,\nu}(\G\,\vert\,\mathscr C)}^{1/n}\norm{\Psi}_{L^{1}(\G\,\vert\,\mathscr C)}^{1/n}  \\
        &=\rho_{L^{1,\nu}(\G\,\vert\,\mathscr C)}(\Psi)\rho_{L^{1}(\G\,\vert\,\mathscr C)}(\Psi).
    \end{align*} In consequence, we get that $\rho_{L^{1,\nu}(\G\,\vert\,\mathscr C)}(\Psi)\leq \rho_{L^{1}(\G\,\vert\,\mathscr C)}(\Psi)$. As the opposite inequality is automatic, we conclude that $L^{1,\nu}(\G\,\vert\,\mathscr C)$ is spectral radius preserving in $L^{1}(\G\,\vert\,\mathscr C)$. The result follows from Barnes' lemma.
\end{proof}

In particular, we have the following corollary. 
\begin{cor}\label{extension}
    Let $\nu$ be a polynomial weight on $\G$. Let $\varphi: L^{1,\nu}(\G\,\vert\,\mathscr C)\to \B$ be a $^*$-homomorphism, with $\B$ a $C^*$-algebra. Then $\varphi$ extends to $L^{1}(\G\,\vert\,\mathscr C)$.
\end{cor}
\begin{proof}
    Let $\Phi\in L^{1,\nu}(\G\,\vert\,\mathscr C)$. We have \begin{align*}
    \norm{\varphi(\Phi)}_\B^2&=\norm{\varphi(\Phi^**\Phi)} \\ &=\rho_\B\big(\varphi(\Phi^**\Phi)\big) \\
    &\leq \rho_{L^{1,\nu}(\G\,\vert\,\mathscr C)}(\Phi^**\Phi) \\
    &=\rho_{L^{1}(\G\,\vert\,\mathscr C)}(\Phi^**\Phi) \\
    &\leq \norm{\Phi^**\Phi}_{L^{1}(\G\,\vert\,\mathscr C)}\leq \norm{\Phi}_{L^{1}(\G\,\vert\,\mathscr C)}^2,
    \end{align*} so $\varphi$ extends to $L^{1}(\G\,\vert\,\mathscr C)$.
\end{proof}

We finally turn our attention to the question of symmetry.

\begin{lem}\label{diff}
    Let $\nu$ be a polynomial weight on $\G$ such that $\nu^{-1}$ belongs to $L^p(\G)$, for $0<p<\infty$. Then there is a constant $A>0$, such that \begin{equation}\label{idk}
        \norm{\Phi}_{L^{1}(\G\,\vert\,\mathscr C)}\leq A \norm{\Phi}_{L^{\infty}(\G\,\vert\,\mathscr C)}^{1/(p+1)}\norm{\Phi}_{L^{1,\nu}(\G\,\vert\,\mathscr C)}^{p/(p+1)},
    \end{equation} for all $\Phi\in L^{1,\nu}(\G\,\vert\,\mathscr C)\cap L^{\infty}(\G\,\vert\,\mathscr C)$.
\end{lem}\begin{proof}
    Let $a>0$ to be determined later and 
    $$
    U_a=\{x\in\G\mid \nu(x)\leq a\}.
    $$ 
    Then for $x\in U_a$, one has $1\leq a^p\nu(x)^{-p}$, so, denoting $B=\norm{\nu^{-1}}_{L^p(\G)}^p$, $$\mu(U_a)=\int_{U_a} 1\d x\leq \int_{U_a} \tfrac{a^p}{\nu(x)^{p}}\d x\leq a^pB.$$ Thus, \begin{align*}
        \norm{\Phi}_{L^{1}(\G\,\vert\,\mathscr C)}&=\int_{U_a} \norm{\Phi(x)}_{\mathfrak C_x}\d x+\int_{\G\setminus U_a} \norm{\Phi(x)}_{\mathfrak C_x}\d x \\
        &\leq \norm{\Phi}_{L^{\infty}(\G\,\vert\,\mathscr C)}\mu(U_a)+\norm{\Phi}_{L^{1,\nu}(\G\,\vert\,\mathscr C)}\sup_{\G\setminus U_a} \tfrac{1}{\nu(x)} \\
        &\leq a^pB\norm{\Phi}_{L^{\infty}(\G\,\vert\,\mathscr C)}+\tfrac{1}{a}\norm{\Phi}_{L^{1,\nu}(\G\,\vert\,\mathscr C)} 
    \end{align*} If we take $$a=\Big(\frac{\norm{\Phi}_{L^{1,\nu}(\G\,\vert\,\mathscr C)}}{B\norm{\Phi}_{L^{\infty}(\G\,\vert\,\mathscr C)}}\Big)^{\tfrac{1}{p+1}},$$ then the result follows, with $A=2B^{\tfrac{1}{p+1}}$.
\end{proof}

The following lemma is inspired by \cite[Lemma 4]{Py73}.

\begin{lem}\label{exist}
    Let $\Phi\in L^{1}(\G\,\vert\,\mathscr C)_{\rm sa}$. There exists a continuous cross-section $\Psi\in C_{\rm c}(\G\,\vert\,\mathscr C)$ such that \begin{equation}
        \rho_{L^{1}(\G\,\vert\,\mathscr C)}(\Phi)\leq \limsup_{n\to \infty}\norm{\Psi^2*\Phi^n}_{L^{1}(\G\,\vert\,\mathscr C)}^{1/n}.
    \end{equation}
\end{lem}
\begin{proof}
    This is obviously true for $\Phi=0$. Otherwise, we consider the sequence 
    $$
    a_n:=\frac{\norm{\Phi^{n+2}}_{L^{1}(\G\,\vert\,\mathscr C)}}{\norm{\Phi^n}_{L^{1}(\G\,\vert\,\mathscr C)}}.
    $$ 
    It is easy to see that $\limsup_{n\to \infty}a_n>0$. Indeed, otherwise would mean that $\lim_{n\to \infty}a_n =0$ and so we can find, for any $\epsilon>0$, an $n_0\in \N$ such that for all $m\geq 1$, one has $a_{n_0+m}<\epsilon$. But this implies that \begin{equation*}
        \frac{\norm{\Phi^{n_0+2m}}_{L^{1}(\G\,\vert\,\mathscr C)}}{\norm{\Phi^{n_0}}_{L^{1}(\G\,\vert\,\mathscr C)}}=\prod_{k=0}^{m-1} a_{n_0+2k}\leq \epsilon^m,
    \end{equation*} and $\lim_{m\to\infty}\norm{\Phi^{n_0+2m}}_{L^{1}(\G\,\vert\,\mathscr C)}^{1/(n_0+2m)}\leq \sqrt{\epsilon}$. We now observe that 
    $$
    0=\lim_{m\to\infty}\norm{\Phi^{n_0+2m}}_{L^{1}(\G\,\vert\,\mathscr C)}^{1/(n_0+2m)}=\rho_{L^1(\G\,\vert\,\mathscr C)}(\Phi)
    $$ 
    and 
    $$
    \rho_{{\rm C^*}(\G\,\vert\,\mathscr C)}(\Phi)\leq\rho_{L^1(\G\,\vert\,\mathscr C)}(\Phi)=0.
    $$ 
    This is, of course, impossible. Now that we established the validity of $\limsup_{n\to \infty}a_n>0$, we let $0<a<\limsup_{n\to \infty}a_n$ and $\Psi\in C_{\rm c}(\G\,\vert\,\mathscr C)$ such that $\norm{\Phi-\Psi}_{L^{1}(\G\,\vert\,\mathscr C)}<\epsilon$, subject to $(2\norm{\Phi}_{L^{1}(\G\,\vert\,\mathscr C)}+\epsilon)\epsilon<\tfrac{a}{2}$ and see that
    $$
    \norm{\Phi^2-\Psi^2}_{L^{1}(\G\,\vert\,\mathscr C)}\leq \norm{\Phi-\Psi}_{L^{1}(\G\,\vert\,\mathscr C)}\big(\norm{\Phi}_{L^{1}(\G\,\vert\,\mathscr C)}+\norm{\Psi}_{L^{1}(\G\,\vert\,\mathscr C)}\big)\leq (2\norm{\Phi}_{L^{1}(\G\,\vert\,\mathscr C)}+\epsilon)\epsilon<\tfrac{a}{2}
    $$ and 
    $$
    \norm{\Psi^2*\Phi^n}_{L^{1}(\G\,\vert\,\mathscr C)}\geq \norm{\Phi^{n+2}}_{L^{1}(\G\,\vert\,\mathscr C)}-\tfrac{a}{2}\norm{\Phi^{n}}_{L^{1}(\G\,\vert\,\mathscr C)}=\norm{\Phi^{n}}_{L^{1}(\G\,\vert\,\mathscr C)}(a_n-\tfrac{a}{2}).
    $$ Since $a_n-\tfrac{a}{2}\geq \tfrac{a}{2}$ for infinitely many $n$, then 
    $$
    \limsup_{n\to \infty}\norm{\Psi^2*\Phi^n}_{L^{1}(\G\,\vert\,\mathscr C)}^{1/n}\geq \lim_{n\to \infty}\norm{\Phi^n}_{L^{1}(\G\,\vert\,\mathscr C)}^{1/n}=\rho_{L^{1}(\G\,\vert\,\mathscr C)}(\Phi),
    $$ finishing the proof.
\end{proof}

\begin{thm}\label{closedness}
    Let $\nu$ be a polynomial weight on $\G$ such that $\nu^{-1}$ belongs to $L^p(\G)$, for some $0<p<\infty$. For $\Phi=\Phi^*\in L^{1,\nu}(\G\,\vert\,\mathscr C)$, one has \begin{equation}\label{closednesseq}
        {\rm Spec}_{L^{1}(\G\,\vert\,\mathscr C)}(\Phi)={\rm Spec}_{L^{1,\nu}(\G\,\vert\,\mathscr C)}(\Phi)={\rm Spec}_{{\rm C^*}(\G\,\vert\,\mathscr C)}\big(\lambda(\Phi)\big),
    \end{equation} in particular, $L^{1,\nu}(\G\,\vert\,\mathscr C)$ is symmetric.
\end{thm}
\begin{proof}
    The first equality is the content of Proposition \ref{L1Lnu}. For the second one, we compute the spectral radius. We will apply Lemmas \ref{diff} and \ref{lemma}. Note that, if $\Psi\in C_{\rm c}(\G\,\vert\,\mathscr C)$ is the cross-section given by Lemma \ref{exist}, then $\Psi^2*\Phi^n,\Psi*\Phi^n\in L^{\infty}(\G\,\vert\,\mathscr C)$ and \begin{align*}
        \norm{\Psi^2*\Phi^n}_{L^{1}(\G\,\vert\,\mathscr C)}&\leq A \norm{\Psi^2*\Phi^n}_{L^{\infty}(\G\,\vert\,\mathscr C)}^{1/(p+1)}\norm{\Psi^2*\Phi^n}_{L^{1,\nu}(\G\,\vert\,\mathscr C)}^{p/(p+1)} \\
        &\leq A \norm{\Psi}_{L^{2}(\G\,\vert\,\mathscr C)}^{1/(p+1)}\norm{\Psi*\Phi^n}_{L^2_\e(\G\,\vert\,\mathscr C)}^{1/(p+1)}\norm{\Phi^n}_{L^{1,\nu}(\G\,\vert\,\mathscr C)}^{p/(p+1)}\norm{\Psi^2}_{L^{1,\nu}(\G\,\vert\,\mathscr C)}^{p/(p+1)} \\
        &\leq A \norm{\Psi}_{L^{2}(\G\,\vert\,\mathscr C)}^{1/(p+1)}\norm{\lambda(\Phi)^n}_{\mathbb B(L^{2}(\G\,\vert\,\mathscr C))}^{1/(p+1)}\norm{\Psi}_{L^2_\e(\G\,\vert\,\mathscr C)}^{1/(p+1)}\norm{\Phi^n}_{L^{1,\nu}(\G\,\vert\,\mathscr C)}^{p/(p+1)}\norm{\Psi^2}_{L^{1,\nu}(\G\,\vert\,\mathscr C)}^{p/(p+1)}
    \end{align*} So taking $n$-th roots and $\limsup$ on $n$ yields \begin{align*}
        \limsup_{n\to\infty}\norm{\Psi^2*\Phi^n}_{L^{1}(\G\,\vert\,\mathscr C)}^{1/n}\leq \rho_{\mathbb B(L^{2}(\G\,\vert\,\mathscr C))}\big(\lambda(\Phi)\big)^{1/(p+1)}\rho_{L^{1,\nu}(\G\,\vert\,\mathscr C)}(\Phi)^{p/(p+1)}
    \end{align*} But $\rho_{L^{1}(\G\,\vert\,\mathscr C)}(\Phi)\leq \limsup_{n\to \infty}\norm{\Psi^2*\Phi^n}_{L^{1}(\G\,\vert\,\mathscr C)}^{1/n}$, by Lemma \ref{exist} and $$\rho_{L^{1}(\G\,\vert\,\mathscr C)}(\Phi)=\rho_{L^{1,\nu}(\G\,\vert\,\mathscr C)}(\Phi), $$ by Proposition \ref{L1Lnu}. Finally, $$\rho_{\mathbb B(L^{2}(\G\,\vert\,\mathscr C))}\big(\lambda(\Phi)\big)=\rho_{\mathbb B_a(L^{2}(\G\,\vert\,\mathscr C))}\big(\lambda(\Phi)\big),$$ by \cite[Corollary 1]{DLW97}. Therefore we have shown that $L^{1,\nu}(\G\,\vert\,\mathscr C)$ is inverse closed in ${\rm C^*}(\G\,\vert\,\mathscr C)$ and hence symmetric.
\end{proof}

\begin{rem}
    When interpreted as a statement about the spectra of elements in $L^{1}(\G\,\vert\,\mathscr C)$, the above theorem implies that, for a locally compact, compactly generated group of polynomial growth, then ${\rm Spec}_{L^{1}(\G\,\vert\,\mathscr C)}(\Phi)\subset \R$, for all $\Phi=\Phi^*\in L^{1,\nu}(\G\,\vert\,\mathscr C)$. It is a very interesting open question whether $L^{1}(\G\,\vert\,\mathscr C)$ is symmetric in such a case.
\end{rem}

We finish this section by noting that the smooth functional calculus defined for $L^{1}(\G\,\vert\,\mathscr C)$ restricts to $L^{1,\nu}(\G\,\vert\,\mathscr C)$. 

\begin{prop}\label{bigO2}
    Let $\G$ be a group of polynomial growth of order $d$, $\nu$ a polynomial weight on $\G$, with $\frac{\nu(xy)}{\nu(x)+\nu(y)}\leq C$. Let also $\Phi=\Phi^*\in C_{\rm c}(\G\,\vert\,\mathscr C)$. Then \begin{equation}
        \norm{u(t\Phi)}_{L^{1,\nu}(\G\,\vert\,\mathscr C)}=O(|t|^{2d+2+4\delta}),\quad\textup{ as }|t|\to\infty,
    \end{equation}
    for $\delta=1+\log_2(C)$.\end{prop}

\begin{proof}
    We proceed exactly as in the proof of Proposition \ref{bigO}, so the simpler details will be skipped. Let $K$ be a compact subset of $\G$, containing both ${\rm Supp}(\Phi)$ and $\e$, and let $n$ be a positive integer. Then
    \begin{equation}\label{decomp2}
        \norm{u(n\Phi)}_{L^{1,\nu}(\G\,\vert\,\mathscr C)}=\int_{K^{n^2-1}}  \norm{u(n\Phi)(x)}_{\mathfrak C_x}\nu(x)\d x+\int_{\G \setminus K^{n^2-1}} \norm{u(n\Phi)(x)}_{\mathfrak C_x}\nu(x)\d x.
    \end{equation}
    The second integral is again bounded easily and we see
    \begin{align*}
        \int_{\G \setminus K^{n^2-1}} \norm{u(n\Phi)(x)}_{\mathfrak C_x}\nu(x)\d x \leq \frac{n^{n^2}}{(n^2)!}e^{n\norm{\Phi}_{L^{1,\nu}(\G\,\vert\,\mathscr C)}}\norm{\Phi}_{L^{1,\nu}(\G\,\vert\,\mathscr C)}^{n^2}=O(1),\quad\textup{ as }n\to\infty.
    \end{align*} 
    Now we turn our attention to the first integral in \eqref{decomp2}. Exactly as before, we have $u(n\Phi)=in\Phi+n\Phi*v(n\Phi)$ and
    \begin{align*}
        \norm{u(n\Phi)(x)}_{\mathfrak C_x}\nu(x)\leq n\norm{\Phi(x)}_{\mathfrak C_x}\nu(x)+n\norm{\Phi}_{L^2(\G\,\vert\,\mathscr C)}\norm{v(n\Phi)}_{L^2_\e(\G\,\vert\,\mathscr C)}\nu(x),
    \end{align*}
    for all $x\in\G$. Moreover, following the reasoning in \cite[Proposition 1]{Py82}, we see that, for $A=\sup_{x\in K}\nu(x)$, one has 
    $$
    \nu(x)\leq A^2(1+n)^{2\delta}, \quad\text{for all }x\in K^n.
    $$
    Finally, we use H\"older's inequality:  
    \begin{align*}
    \int_{K^{n^2-1}}  \norm{u(n\Phi)(x)}_{\mathfrak C_x}\nu(x)\d x &\leq n\norm{\Phi}_{L^{1,\nu}(\G\,\vert\,\mathscr C)}+A^2n^{4\delta+1}\norm{\Phi}_{L^2(\G\,\vert\,\mathscr C)}\norm{v(n\Phi)}_{L^2_\e(\G\,\vert\,\mathscr C)}\mu(K^{n^2-1}) \\
    &\overset{\eqref{dixlem}}{\leq} n\Big(\norm{\Phi}_{L^{1,\nu}(\G\,\vert\,\mathscr C)}+\tfrac{A^2}{2}n^{4\delta+1}\mu(K^{n^2-1})\norm{\Phi}_{L^{2}(\G\,\vert\,\mathscr C)}^2\Big).
    \end{align*} 
    Thus we have shown that $\norm{u(n\Phi)}_{L^{1,\nu}(\G\,\vert\,\mathscr C)}=O(n^{2d+2+4\delta})$, as $n\to\infty$. The final step of the proof is as in Proposition \ref{bigO}.
\end{proof}

As before, we obtain the existence of a smooth functional calculus for the self-adjoint cross-sections $\Phi\in C_{\rm c}(\G\,\vert\,\mathscr C)$. 

\begin{thm}\label{funccalc2}
     Let $\G$ be a group of polynomial growth of order $d$, $\nu$ a polynomial weight on $\G$ and $\Phi=\Phi^*\in C_{\rm c}(\G\,\vert\,\mathscr C)$. Let $f\in C_{\rm c}^{\infty}(\R)$ with Fourier transform $\widehat f$. Then \begin{enumerate}
         \item[(i)] $f(\Phi)=\frac{1}{2\pi}\int_{\mathbb R} \widehat{f}(t)e^{it\Phi}\d t$ exists in $\widetilde{L^{1,\nu}(\G\,\vert\,\mathscr C)}$, or in $L^{1,\nu}(\G\,\vert\,\mathscr C)$ if $f(0)=0$.
         \item[(ii)] For any non-degenerate $^*$-representation $\Pi:L^{1,\nu}(\G\,\vert\,\mathscr C)\to\BofH$, we have $$\widetilde{\Pi}(f\big(\Phi\big))=f\big(\Pi(\Phi)\big).$$ Furthermore, if $\Pi$ is injective, we also have $${\rm Spec}_{\widetilde{L^{1,\nu}(\G\,\vert\,\mathscr C)}}\big(f(\Phi)\big)={\rm Spec}_{\BofH}\big(\widetilde\Pi\big(f(\Phi)\big)\big).$$ 
     \end{enumerate}
\end{thm}

\section{Some consequences}\label{consequences}

From now on, $\mathfrak D$ will denote either $L^{1}(\G\,\vert\,\mathscr C)$ or $L^{1,\nu}(\G\,\vert\,\mathscr C)$, where $\nu$ is a polynomial weight such that $\nu^{-1}\in L^p(\G)$. $\G$ is assumed to have polynomial growth of order $d$. We remark that compact generation of $\G$ is not necessarily assumed. If needed, we recall that sufficient conditions for the symmetry of $\mathfrak D$ were given in Remark \ref{symm} and Theorem \ref{closedness}.

\begin{defn}
    Let $\B$ be a semisimple commutative Banach algebra with spectrum $\Delta$. $\B$ is called {\it regular} if for every closed set $X\subset \Delta$ and every point $\omega\in \Delta\setminus X$, there exists an element $b\in \B$ such that $\hat b(\varphi)=0$ for all $\varphi\in X$ and $\hat b(\omega)\not=0$.
\end{defn}

\begin{lem}\label{spinvariance}
    Let $\B$ be a Banach $^*$-algebra and $b\in \B$ be self-adjoint, with polynomial growth of order $d$. Suppose $\Pi:\B \to\BofH$ be any non-degenerate $^*$-representation such that $\widetilde\Pi|_{\B(b,1)}$ is injective. If $\varphi_b:C_{\rm c}^{d+2}(\mathbb R)\to \B(b,1)$ is the $^*$-homomorphism in Theorem \ref{DixBai} and $a\in {\rm Im}\varphi_b$, then $${\rm Spec}_{\B(b,1)}(a)={\rm Spec}_{\B}(a)={\rm Spec}_{\BofH}\big(\widetilde\Pi(a)\big).$$
    In particular, $\B(b,1)$ is regular.
\end{lem}
\begin{proof}
    It is immediate that ${\rm Spec}_{\BofH}\big(\widetilde\Pi(a)\big)\subset {\rm Spec}_{\B}(a)={\rm Spec}_{\widetilde\B}(a)\subset {\rm Spec}_{\B(b,1)}(a)$. To show the reversed containments, it is enough to show that, if $\widetilde\Pi(a)$ is invertible in $\BofH$, then $a$ is invertible in $\B(b,1)$. Indeed, let $q:C_{\rm c}^{d+2}(\mathbb R)\to C({\rm Spec}_{\BofH}\big(\widetilde\Pi(a)\big))$ be the restriction map. Theorem \ref{DixBai} implies that the diagram $$
\begin{tikzcd}
[arrows=rightarrow]
    C_{\rm c}^{d+2}(\mathbb R)\arrow{r}{\varphi_b}\arrow{d}{q} & \B(b,1)
    \arrow{d}{\widetilde\Pi}
    \\
    C({\rm Spec}_{\BofH}\big(\Pi(b)\big))
    \arrow{r}{\cong}
    & {\rm C^*}(\Pi(b),1)
\end{tikzcd}
$$ commutes. In particular, if $a=\varphi_b(f)$ and using the spectral mapping theorem,  $$f\big({\rm Spec}_{\BofH}\big(\Pi(b)\big)\big)={\rm Spec}_{\BofH}\big(f\big(\Pi(b)\big)\big)\overset{\ref{DixBai}(ii)}{=}{\rm Spec}_{\BofH}\big(\widetilde{\Pi}\big(f(b)\big)\big)={\rm Spec}_{\BofH}\big(\widetilde{\Pi}(a)\big).$$ If $\widetilde{\Pi}(a)$ is invertible, then $0\not \in f\big({\rm Spec}_{\BofH}\big(\Pi(b)\big)\big)$. So there exists a function $g\in C_{\rm c}^{d+2}(\mathbb R)$ such that $f(x)g(x)=1$, for $x\in {\rm Spec}_{\BofH}\big(\Pi(b)\big)$. Therefore, $\varphi_b(f\cdot g)=1$ and $a$ is invertible in $\B(b,1)$, with inverse $a^{-1}=g(b)$.
The statement about regularity then follows from the fact that any compact set in $\CC$ is completely regular, but the separating functions can be chosen to be smooth.
\end{proof}

The following definition is due to Barnes \cite{Ba81} and gives a useful criterion for $^*$-regularity, among other properties.
\begin{defn}
    A reduced Banach $^*$-algebra $\B$ is called {\it locally regular} if there is a subset $R\subset \B_{\rm sa}$, dense in $\B_{\rm sa}$ and such that $\B(b)$ is regular, for all $b\in R$.
\end{defn}

\begin{cor}\label{locallyreg}
    $\mathfrak D$ is locally regular.
\end{cor}

\subsection{Preservation of spectra and \texorpdfstring{$^*$}--regularity}

Let $\Pi:\mathfrak D\to \mathbb B(\mathcal X)$ be a representation of $\mathfrak D$ on the Banach space $\mathcal X$. The idea of this subsection is to understand the spectrum of $\Pi(\Phi)$, at least for self-adjoint $\Phi$. This is particularly important (and also easier) in the case of $^*$-representations as it allows us to understand the $C^*$-norms in $\mathfrak D$. We also provide applications to the ideal theory of $\mathfrak D$. 

If $\B$ is a reduced Banach $^*$-algebra, then $\iota_\B:\B\to {\rm C^*}(\B)$ will denote its canonical embedding into the enveloping $C^*$-algebra ${\rm C^*}(\B)$. The spaces ${\rm Prim}\, {\rm C^*}(\B)$, ${\rm Prim}\B$ and ${\rm Prim}_*\B$ denote, respectively, the space of primitive ideals of ${\rm C^*}(\B)$, the space of primitive ideals of $\B$ and the space of kernels of topologically irreducible
$^*$-representations of $\B$, all of them equipped with the Jacobson topology. It is known that $\iota_\B$ induces a continuous surjection ${\rm Prim}\, {\rm C^*}(\B)  \to{\rm Prim}_*\B$ \cite[Corollary
10.5.7]{Pa94}. We recall that for a subset $S\subset \B$, its hull corresponds to $$h(S)=\{I\in {\rm Prim}_*\B\mid S\subset I\},$$ while the kernel of a subset $C\subset {\rm Prim}_*\B$ is $$k(C)=\bigcap\{I\mid I\in C\}.$$ 

\begin{defn}
    Let $\B$ be a reduced Banach $^*$-algebra. \begin{enumerate}
        \item[(i)] $\B$ is called $C^*$-unique if there is a unique $C^*$-norm on $\B$.
        \item[(ii)] $\B$ is called $^*$-regular if the surjection ${\rm Prim}\, {\rm C^*}(\B)\to {\rm Prim}_*\B$ is a homeomorphism.
    \end{enumerate}
\end{defn} 

It is well-known that $^*$-regularity implies $C^*$-uniqueness. In fact, the latter may be rephrased as the following: $\B$ is $C^*$-unique if and only if, for every closed ideal $I\subset {\rm C^*}(\B)$, one has $I\cap \B\not=\{0\}$. It might also be referred to as the 'ideal intersection property'. The next theorem shows that $^*$-regularity is much stronger.

\begin{thm}\label{regularity}
    $\mathfrak D$ is $^*$-regular. In particular, the following are true. \begin{enumerate}
        \item[(i)] For any pair of \,$^*$-representations $\Pi_1, \Pi_2$ of $\mathfrak D$ the inclusion ${\rm ker}\,\Pi_1\subset {\rm ker}\,\Pi_2$ implies $\norm{\Pi_2(\Phi)}\leq \norm{\Pi_1(\Phi)}$, for all $\Phi\in \mathfrak D$.
        \item[(ii)] Let $I$ be a closed ideal of ${\rm C^*}(\G\,\vert\,\mathscr C)$, then $\overline{\mathfrak D\cap I}^{\norm{\cdot}_{{\rm C^*}(\G\,\vert\,\mathscr C)}}=I$.
        \item[(iii)] If $I\in {\rm Prim}_*\, \mathfrak D$, then $\overline{I}^{\norm{\cdot}_{{\rm C^*}(\G\,\vert\,\mathscr C)}}\in {\rm Prim}\, {\rm C^*}(\G\,\vert\,\mathscr C)$. If $\mathfrak D$ is symmetric, then for any $I\in {\rm Prim}\, \mathfrak D$, we have $\overline{I}^{\norm{\cdot}_{{\rm C^*}(\G\,\vert\,\mathscr C)}}\in {\rm Prim}\, {\rm C^*}(\G\,\vert\,\mathscr C)$.
    \end{enumerate}
\end{thm}
\begin{proof}
    Any $C^*$-norm must coincide with the universal $C^*$-norm on the elements of polynomial growth, due to Lemma \ref{spinvariance}. But for $\mathfrak D$ these elements are dense in $\mathfrak D_{\rm sa}$, therefore such a norm must coincide with the universal $C^*$-norm on all of $\mathfrak D$. The same is true for any quotient of $\mathfrak D$, since the homomorphic image of an element of polynomial growth also has polynomial growth. Then $^*$-regularity follows from \cite[Theorem 10.5.18]{Pa94}. The rest of the assertions are consequences of $^*$-regularity. Indeed, \emph{(i)} follows from \cite[Satz 2]{BLSV78} and \emph{(ii)} from \cite[Theorem 10.5.19]{Pa94}. 
    
    We now prove \emph{(iii)}. If $I\in {\rm Prim}\, \mathfrak D$, then there exists an irreducible $^*$-representation $\Pi:\mathfrak D\to\BofH$ with $I={\rm ker}\, \Pi$. Let $\widetilde\Pi$ be the unique extension of $\Pi$ to ${\rm C^*}(\G\!\mid\!\mathscr C)$. Then $I=\mathfrak D\cap {\rm ker}\, \widetilde\Pi$, so by the previous point, ${\rm ker}\, \widetilde\Pi=\overline{I}^{\norm{\cdot}_{{\rm C^*}(\G\,\mid\,\mathscr C)}}\in {\rm Prim}\, {\rm C^*}(\G\!\mid\!\mathscr C)$. Now, if $\mathfrak D$ is symmetric, then ${\rm Prim}\, \mathfrak D\subset {\rm Prim}_*\, \mathfrak D$ (cf. \cite[page 39]{Le73}).
    \end{proof}

We finish the subsection with a result involving general spectral invariance under homomorphisms to Banach algebras. The price we pay for this generality is assuming symmetry. 

\begin{thm}
    Suppose $\mathfrak D$ is symmetric. Let $\B$ be a unital Banach algebra and $\varphi:\widetilde{\mathfrak D}\to\B$ a continuous unital homomorphism. Set $I={\rm ker}\,\varphi$. \begin{enumerate}
        \item[(i)] If $\Phi \in \widetilde{\mathfrak D}$ is normal, $${\rm Spec}_{\widetilde{\mathfrak D}/I}(\Phi+I)={\rm Spec}_{\B}\big(\varphi(\Phi)\big).$$ 
        \item[(ii)] For a general $\Phi \in \widetilde{\mathfrak D}$, $${\rm Spec}_{\widetilde{\mathfrak D}/I}(\Phi+I)={\rm Spec}_{\B}\big(\varphi(\Phi)\big)\cup \overline{{\rm Spec}_{\B}\big(\varphi(\Phi^*)\big)}.$$
    \end{enumerate}
\end{thm}
\begin{proof}
    This is an application of \cite[Theorem 2.2]{Ba87}. While symmetry is assumed, $\mathfrak D$ is $^*$-quotient inverse closed because of local regularity (see \cite[Theorem 2.1]{Ba87}). 
\end{proof} 

An immediate (but remarkable) application of the latter theorem gives the following result.
\begin{cor}
    For $p\in[1,\infty]$, let $\lambda_p:\widetilde{\mathfrak D}\to\mathbb B(L^{p}(\G\,\vert\,\mathscr C))$ be the representation given by $\lambda_p(\Phi)\Psi=\Phi*\Psi$. Suppose that $\mathfrak D$ is symmetric. Then 
    \begin{enumerate}
        \item[(i)] If $\Phi \in \widetilde{\mathfrak D}$ is normal, $${\rm Spec}_{\widetilde{\mathfrak D}}(\Phi)={\rm Spec}_{\mathbb B(L^{p}(\G\,\vert\,\mathscr C))}\big(\lambda_p(\Phi)\big).$$ 
        \item[(ii)] For a general $\Phi \in \widetilde{\mathfrak D}$, $${\rm Spec}_{\widetilde{\mathfrak D}}(\Phi)={\rm Spec}_{\mathbb B(L^{p}(\G\,\vert\,\mathscr C))}\big(\lambda_p(\Phi)\big)\cup \overline{{\rm Spec}_{\mathbb B(L^{p}(\G\,\vert\,\mathscr C))}\big(\lambda_p(\Phi^*)\big)}.$$
    \end{enumerate}
\end{cor}
We remark that, in the case $p=1$, one always has ${\rm Spec}_{\widetilde{\mathfrak D}}(\Phi)={\rm Spec}_{\mathbb B(L^{1}(\G\,\vert\,\mathscr C))}\big(\lambda_1(\Phi)\big),$ with no assumptions of symmetry whatsoever. This is trivial for $\mathfrak D=L^{1}(\G\,\vert\,\mathscr C)$ and follows from Proposition \ref{L1Lnu} for $\mathfrak D=L^{1,\nu}(\G\,\vert\,\mathscr C)$.

\subsection{Minimal ideals of a given hull}

Now we turn our attention to the following problem: Study the existence of minimal ideals of $\mathfrak D$ with a given hull in ${\rm Prim }_*\mathfrak D$. We are able to positively answer this problem for $\mathfrak D$ under the assumption of symmetry. Our references here are \cite{DLM04} and, especially, \cite{Lu80}. For a given closed subset $C\subset {\rm Prim }_*\mathfrak D$, we will use the following notation: for $\Phi\in\mathfrak D$, we set \begin{equation*}
    \norm{\Phi}_C=\sup _{{\rm ker}\,\Pi \in C}\norm{\Pi(\Phi)},
\end{equation*} with $\norm{\Phi}_\emptyset=0$ and
\begin{align*}
    m(C)=&\{f(\Phi) \mid \Phi \in C_{\rm c}(\G\,\vert\,\mathscr C)_{\rm sa},\norm{\Phi}_{L^{1}(\G\,\vert\,\mathscr C)} \leq 1, \\
    &\hspace{80pt}f\in C_{\rm c}^{\infty}(\mathbb R)
, f\equiv 0 \textup{ on a neighborhood of }[-\norm{\Phi}_C,\norm{\Phi}_C]\}.
\end{align*}
We let $j(C)$ be the closed two-sided ideal of $\mathfrak D$ generated by $m(C)$. Note that for $C=\emptyset$, we have
\begin{align*}
m(\emptyset)=&\{f(\Phi) \mid \Phi \in C_{\rm c}(\G\,\vert\,\mathscr C)_{\rm sa},\norm{\Phi}_{L^{1}(\G\,\vert\,\mathscr C)} \leq 1, \\
    &\hspace{150pt}f\in C_{\rm c}^{\infty}(\mathbb R)
, f\equiv 0 \textup{ on a neighborhood of }0\}.
\end{align*}

\begin{lem}
    The hull of $j(C)$ is $C$.
\end{lem}

\begin{proof}
    We first prove that $C \subset h(j(C))$. Indeed, let $\Pi$ be a $^*$-representation with ${\rm ker}\,\Pi \in C$ and $f(\Phi) \in m(C)$. Then $\Pi\big(f(\Phi)\big)=f\big(\Pi(\Phi)\big)=0$, since $f\equiv 0$ on the spectrum of $\Pi(\Phi)$. Hence $m(C) \subset {\rm ker}\, \Pi$ and therefore ${\rm ker}\, \Pi \in h(j(C))$ and $C \subset h(j(C))$.

    On the other hand, let $\Pi$ be a $^*$-representation with ${\rm ker}\,\Pi\in {\rm Prim }_*\mathfrak D\setminus C$, then, because of Theorem \ref{regularity} there exists $\Phi \in C_{\rm c}(\G\,\vert\,\mathscr C)$ such that $\norm{\Phi}_C<\norm{\Pi(\Phi)}$. In fact, $\Phi$ can be chosen to be self-adjoint and have $\norm{\Phi}_{L^{1}(\G\,\vert\,\mathscr C)} \leq 1$. In particular, we can find a function $f\in C_{\rm c}^{\infty}(\mathbb R)$ such that $f\equiv 0$ on a neighborhood of $[-\norm{\Phi}_C,\norm{\Phi}_C]$ and $f(\norm{\Pi(\Phi)})\not= 0$. Since $\Phi$ is self-adjoint, $\norm{\Pi(\Phi)}$ lies in the spectrum of $\Pi(\Phi)$ and $$\Pi\big(f(\Phi)\big)=f\big(\Pi(\Phi)\big)\not=0,$$ from which it follows that ${\rm ker}\,\Pi\not \in h(j(C))$.
\end{proof}

\begin{thm}\label{minideals}
    Suppose $\mathfrak D$ is symmetric and let $C$ be a closed subset of ${\rm Prim }_*\mathfrak D$. There exists a closed two-sided ideal $j(C)$ of $\mathfrak D$, with $h(j(C))=C$, which is contained in every two-sided closed ideal I with $h(I) \subset C$.
\end{thm}

\begin{proof}
    Take $f(\Phi) \in m(C)$ arbitrary and choose $g\in C_{\rm c}^{\infty}(\mathbb R)$ such that $g\equiv 1$ in ${\rm Supp}(f)$ and $g\equiv 0$ in a neighborhood of $[-\norm{\Phi}_C,\norm{\Phi}_C]$. Thus $$f(\Phi)=(f\cdot g)(\Phi)=f(\Phi)*g(\Phi)$$ and $g(\Phi)\in m(C)$. This implies that $\mathfrak D$ satisfies the conditions of \cite[Lemma 2]{Lu80} and the conclusion follows.
\end{proof}

\subsection{The Wiener property}\label{wiener}

We now consider the following property, which is intended as an abstract generalization of Wiener's tauberian theorem.

\begin{defn}
    Let $\B$ be a Banach $^*$-algebra. We say that $\B$ has the Wiener property $(W)$ if for every proper closed two-sided ideal $I\subset\B$, there exists a topologically irreducible $^*$-representation $\Pi:\B\to\BofH$, such that $I \subset {\rm ker}\,\Pi$. 
\end{defn}

During this subsection we fix a bounded approximate identity $(\Psi_{\alpha})_{\alpha}$ in $L^{1,\nu}(\G\,\vert\,\mathscr C)$ such that, for all $\alpha$, $\Psi_{\alpha}$ is continuous, \begin{equation}\label{restr}
    \Psi_{\alpha}=\Psi_{\alpha}^{*}, \quad\norm{\Psi_{\alpha}}_{L^{1,\nu}(\G\,\vert\,\mathscr C)} \leq B, \quad\norm{\Psi_{\alpha}}_{L^{1}(\G\,\vert\,\mathscr C)} \leq 1, \quad {\rm Supp}(\Psi_{\alpha}) \subset V_{\alpha} \subset K,
\end{equation} where $B$ is a positive constant, $V_{\alpha}$ a compact symmetric neighborhood of $\e$ in $\G$ and $K$ a fixed compact set. Our objective in this subsection will be to prove the following lemma and derive the Wiener property of $\mathfrak D$ as a consequence. First we comment on the existence of the bounded approximate identity we just fixed.

\begin{rem}
    The existence of a bounded approximate identity satisfying \eqref{restr} is clear in the case of $L^{1,\nu}(\G)$. For the general case, one may construct one by adapting the proof of \cite[Theorem 5.11]{FD88} using the bounded approximate identity of $L^{1,\nu}(\G)$ satisfying \eqref{restr} as the strong approximate identity of $\G$.
\end{rem}

\begin{lem}\label{appidentity}
    Let $f\in C_{\rm c}^{\infty}(\mathbb R)$ so that $f(0)=0$ and $f(1)=1$. Then 
    $$
    \lim_{\alpha} \,\norm{\Phi_0*f(\Psi_{\alpha})*\Phi_1-\Phi_0*\Phi_1}_{\mathfrak D}=0,
    $$
    for all cross-sections $\Phi_0,\Phi_1\in C_{\rm c}(\G\,\vert\,\mathscr C)$.
\end{lem}
\begin{proof}
    Without loss of generality (and for simplicity), we assume that $B\geq1$. For any function $f\in C_{\rm c}^{\infty}(\mathbb R)$ with $0=f(0)=\int_{\mathbb R}\widehat{f}(t)\d t$ and $1=f(1)=\frac{1}{\sqrt{2\pi}}\int_{\mathbb R}\widehat{f}(t)e^{it}\d t$, one sees that 
    \begin{align*}
    \norm{\Phi_0*f(\Psi_{\alpha})*\Phi_1-\Phi_0*\Phi_1}_{\mathfrak D}&=\frac{1}{\sqrt{2\pi}} \Big\|\int_{\mathbb R}\widehat{f}(t)\Phi_0*\big(u(t\Psi_{\alpha})-e^{it}\big)*\Phi_1\, \d t\Big\|_{\mathfrak D}  \\
    &= \frac{1}{\sqrt{2\pi}}\Big\|{\int_{\mathbb R}\widehat{f}(t)\Phi_0*\big(e^{it\Psi_{\alpha}}-e^{it}\big)*\Phi_1\, \d t\Big\|}_{\mathfrak D}.
\end{align*} 
Let $R>0$ be fixed and let us argue that
\begin{equation}\label{goal2}
    \lim_{\alpha}\Big\|\int_{|t|\leq R}\widehat{f}(t)\Phi_0*\big(e^{it\Psi_{\alpha}}-e^{it}\big)*\Phi_1\, \d t\Big\|_{\mathfrak D} =0.
\end{equation}
Indeed, it follows from the computation 
\begin{align*}
    \Big\|\int_{|t|\leq R}\widehat{f}(t)\Phi_0*\big(e^{it\Psi_{\alpha}}-e^{it}\big)*\Phi_1\, \d t\Big\|_{\mathfrak D}&\leq  \norm{\widehat{f}}_{C_0(\R)}\norm{\Phi_0}_{\mathfrak D}\int_{|t|\leq R}\big\|\big(e^{it\Psi_{\alpha}}-e^{it}\big)*\Phi_1\big\|_{\mathfrak D}\d t \\
    &\leq 2R \norm{\widehat{f}}_{C_0(\R)} \norm{\Phi_0}_{\mathfrak D}\sum_{k\in\mathbb N} \frac{R^k}{k!}\norm{\Psi_{\alpha}^k*\Phi_1-\Phi_1}_{\mathfrak D} \\
    &\leq 2 R \norm{\widehat{f}}_{C_0(\R)} \norm{\Phi_0}_{\mathfrak D}\norm{\Psi_{\alpha}*\Phi_1-*\Phi_1}_{\mathfrak D}\sum_{k\in\mathbb N}\frac{R^k}{k!}\sum_{j=0}^{k-1} B^j \\
    &\leq 2 BR^2e^{BR} \norm{\widehat{f}}_{C_0(\R)}\norm{\Psi_{\alpha}*\Phi_1-\Phi_1}_{\mathfrak D}.
\end{align*}
We now turn our attention to the rest of the integral. We see that
\begin{align*}
\Big\|\int_{|t|>R}\widehat{f}(t)\Phi_0*\big(e^{it\Psi_{\alpha}}&-e^{it}\big)*\Phi_1\, \d t\Big\|_{\mathfrak D}\leq  \\
    &\int_{|t|>R}|\widehat{f}(t)|\norm{\Phi_0*e^{it\Psi_{\alpha}}*\Phi_1 }_{\mathfrak D}\d t+\norm{\Phi_0}_{\mathfrak D}\norm{\Phi_1}_{\mathfrak D}\Big|\int_{|t|>R}\widehat{f}(t)e^{it}\, \d t\Big|.
\end{align*}
So it is enough to find $R>0$ such that, given any $\alpha$ and any $\epsilon>0$, one has
\begin{equation}\label{goal}
    \int_{|t|>R}|\widehat{f}(t)|\norm{\Phi_0*e^{it\Psi_{\alpha}}*\Phi_1 }_{\mathfrak D}\d t<\frac{\epsilon}{2}.
\end{equation}
That is because, in such a case, it is easy to enlarge $R$ to also guarantee
$$
\Big|\int_{|t|>R}\widehat{f}(t)e^{it}\, \d t\Big|<\frac{\epsilon}{2\norm{\Phi_0}_{\mathfrak D}\norm{\Phi_1}_{\mathfrak D}}
$$
and that, combined with \eqref{goal2}, will yield the result.

What follows now is an estimation of the growth of $\norm{\Phi_0*e^{it\Psi_{\alpha}}*\Phi_1 }_{\mathfrak D}$ so we will proceed more or less as in Proposition \ref{bigO2}. We will complete the proof for the case $\mathfrak D=L^{1,\nu}(\G\,\vert\,\mathscr C)$, which is the most complicated case anyways.

Let $U$ be a compact subset of $\G$ containing $K\cup {\rm Supp}(\Phi_0)\cup {\rm Supp}(\Phi_1)\cup\{\e\}$ and let $n\in\N$. Then the convolution $\Phi_0*e^{it\Psi_{\alpha}}*\Phi_1$ happens in $\bigcup_{n\in\N} U^n$ and we have
    \begin{equation}\label{decomp3}
        \norm{\Phi_0*e^{it\Psi_{\alpha}}*\Phi_1 }_{L^{1,\nu}(\G\,\vert\,\mathscr C)}=\int_{U^{n}}  \norm{\Phi_0*e^{it\Psi_{\alpha}}*\Phi_1(x) }_{\mathfrak C_x}\nu(x)\d x+\int_{\G \setminus U^{n}} \norm{\Phi_0*e^{it\Psi_{\alpha}}*\Phi_1 (x)}_{\mathfrak C_x}\nu(x)\d x.
    \end{equation}
We then define, for $x\in \bigcup_{n\in\N} U^n$, the quantity $\si_U(x)=\inf\{n\in\N\mid x\in U^n\}$. Reasoning as in \cite[Proposition 1]{Py82} we find constants $A,\delta>0$ so that
$$
\nu(x)\leq A(1+\si_U(x))^{\delta},\quad \text{ for all }x\in \bigcup_{n\in\N} U^n.
$$
The first integral in \eqref{decomp3} can be then bounded by
\begin{align*}
    \int_{U^{n}}  \norm{\Phi_0*e^{it\Psi_{\alpha}}*\Phi_1(x)}_{\mathfrak C_x}\nu(x)\d x &\leq A(1+n)^{\delta}\mu(U^{n}) \norm{\Phi_0*e^{it\Psi_{\alpha}}*\Phi_1}_{L^{\infty}(\G\,\vert\,\mathscr C)}\\
    &\leq A(1+n)^{\delta}\mu(U^{n}) \norm{\Phi_0}_{L^2(\G\,\vert\,\mathscr C)}\norm{e^{in\Psi_{\alpha}}*\Phi_1}_{L^2_\e(\G\,\vert\,\mathscr C)} \\
    &\leq A(1+n)^{\delta}\mu(U^{n}) \norm{\Phi_0}_{L^2(\G\,\vert\,\mathscr C)}\norm{\Phi_1}_{L^2(\G\,\vert\,\mathscr C)}.
\end{align*} 
On the other hand, observe that $\nu'(x)=e^{\si_U(x)}$ is at least $1$ and satisfies $\nu'(xy)\leq \nu'(x)\nu'(y)$ and $E=\sup_{x\in K}\nu'(x)<\infty$. We also note that for all $x\in (\bigcup_{k\in \N} U^k)\setminus U^n$, we have $\nu'(x)\geq e^n$. This implies
\begin{align*}
        \int_{\G \setminus U^{n}} \norm{\Phi_0*e^{it\Psi_{\alpha}}*\Phi_1(x)}_{\mathfrak C_x}\nu(x)\d x &\leq e^{-n}\norm{\Phi_0*e^{it\Psi_{\alpha}}*\Phi_1}_{L^{1,\nu\nu'}(\G\,\vert\,\mathscr C)} \\
        &\leq \norm{\Phi_0}_{L^{1,\nu\nu'}(\G\,\vert\,\mathscr C)}\norm{\Phi_1}_{L^{1,\nu\nu'}(\G\,\vert\,\mathscr C)}e^{-n+|t|\norm{\Psi_\alpha}_{L^{1,\nu\nu'}(\G\,\vert\,\mathscr C)}} \\
        &\leq \norm{\Phi_0}_{L^{1,\nu\nu'}(\G\,\vert\,\mathscr C)}\norm{\Phi_1}_{L^{1,\nu\nu'}(\G\,\vert\,\mathscr C)}e^{-n+BE|t|}.
\end{align*}
Therefore, choosing $n=\ceil{BE|t|}$ yields that
$$
\norm{\Phi_0*e^{it\Psi_{\alpha}}*\Phi_1 }_{L^{1,\nu}(\G\,\vert\,\mathscr C)}=O(t^{\delta+d})+O(1)=O(t^{d+\delta}),\quad \text{ as }|t|\to\infty.
$$
Finally, we can use the fact that $f$ is smooth and compactly supported to find a constant $C'>0$ so that $|\widehat{f}(t)|\leq\frac{C'}{(1+|t|)^{d+\delta+2}}$ for all $t\in\R$. This guarantees the convergence of the integral $\int_{\R}|\widehat{f}(t)|\norm{\Phi_0*e^{it\Psi_{\alpha}}*\Phi_1 }_{\mathfrak D}\d t$. Hence we can find an $R>0$ that satisfies \eqref{goal} and the proof is finished.
\end{proof}

With Lemma \ref{appidentity} at hand, we are finally able to derive the Wiener property for $\mathfrak D$.

\begin{thm}
    Suppose $\mathfrak D$ is symmetric, then it has the Wiener property.
\end{thm}

\begin{proof}
    By Lemma \ref{appidentity}, $j(\emptyset)$ contains all products of elements in $C_{\rm c}(\G\,\vert\,\mathscr C)$ and hence $j(\emptyset)=\mathfrak D$. Now, if $I$ is a closed two-sided ideal of $\mathfrak D$ such that $h(I)=\emptyset$, then $\mathfrak D=j(\emptyset) \subset I$ by Theorem \ref{minideals}.
\end{proof}

\subsection{Norm-controlled inversion}\label{normcontrolled}

The purpose of this section is to produce a dense Banach $^*$-subalgebra of $L^{1}(\G\,\vert\,\mathscr C)$, which is not only symmetric, but it also admits a norm-controlled inversion (at least in the discrete/unital case) in ${\rm C^*}(\G\,\vert\,\mathscr C)$. In fact, because of the results in \cite{Ni99}, it seems unreasonable to expect norm-controlled inversion in all of $L^{1}(\G\,\vert\,\mathscr C)$. For that reason, and motivated by the results in \cite{SaSh19}, we will consider the space $\mathfrak E=L^{1,\nu}(\G\,\vert\,\mathscr C)\cap L^\infty(\G\,\vert\,\mathscr C)$, endowed with the norm $$\norm{\Phi}_{\mathfrak E}=\max\{\norm{\Phi}_{L^{1,\nu}(\G\,\vert\,\mathscr C)},\norm{\Phi}_{L^\infty(\G\,\vert\,\mathscr C)}\}.$$ If $\G$ is discrete, $\mathfrak E$ coincides with $\ell^{1,\nu}(\G\,\vert\,\mathscr C)$, whereas for compact $\G$, we have $\mathfrak E=L^\infty(\G\,\vert\,\mathscr C)$. 

\begin{prop}\label{gendiff1}
    $\mathfrak E$ is a symmetric Banach $^*$-subalgebra of $L^{1}(\G\,\vert\,\mathscr C)$. Moreover, there exists a constant $D\geq 1$ such that \begin{equation}\label{gendiff}
        \norm{\Phi^4}_{\mathfrak E}\leq D \norm{\Phi}_{L^{2}_\e(\G\,\vert\,\mathscr C)}^{1/(p+1)}\norm{\Phi}_{\mathfrak E}^{(4p+3)/(p+1)}
    \end{equation} holds for all $\Phi\in\mathfrak E$.
\end{prop}

\begin{proof}
    Because of Young's inequality, $\Phi*\Psi\in\mathfrak E$ and $\Phi^*\in\mathfrak E$ as soon as $\Phi,\Psi\in\mathfrak E$. It is also clear that $\mathfrak E$ is a Banach space. Moreover, \begin{align*}
        \norm{\Phi*\Psi}_{\mathfrak E}&= \max\{\norm{\Phi*\Psi}_{L^{1,\nu}(\G\,\vert\,\mathscr C)},\norm{\Phi*\Psi}_{L^{\infty}(\G\,\vert\,\mathscr C)}\} \\
        &\leq \norm{\Phi}_{L^{1,\nu}(\G\,\vert\,\mathscr C)}\max\{\norm{\Psi}_{L^{1,\nu}(\G\,\vert\,\mathscr C)},\norm{\Psi}_{L^{\infty}(\G\,\vert\,\mathscr C)}\}\\
        &\leq\norm{\Phi}_{\mathfrak E}\norm{\Psi}_{\mathfrak E}
    \end{align*} so $\mathfrak E$ is a Banach $^*$-subalgebra of $L^{1}(\G\,\vert\,\mathscr C)$. It is worth noting that $ \norm{\Phi}_{L^{2}(\G\,\vert\,\mathscr C)}\leq  \norm{\Phi}_{\mathfrak E}$. For showing symmetry, we will compute the spectral radius. In particular we note that, for $\Phi\in\mathfrak E$, \begin{align*}
        \norm{\Phi^4}_{L^{1,\nu}(\G\,\vert\,\mathscr C)}&\overset{\eqref{diffeq}}{\leq} 2C\norm{\Phi^2}_{L^{1}(\G\,\vert\,\mathscr C)}\norm{\Phi^2}_{L^{1,\nu}(\G\,\vert\,\mathscr C)} \\ 
        &\overset{\eqref{idk}}{\leq}  2CA \norm{\Phi^2}_{L^{\infty}(\G\,\vert\,\mathscr C)}^{1/(p+1)}\norm{\Phi^2}_{L^{1,\nu}(\G\,\vert\,\mathscr C)}^{(2p+1)/(p+1)} \\
        &\leq  2CA \norm{\Phi}_{L^{2}_\e(\G\,\vert\,\mathscr C)}^{1/(p+1)}\norm{\Phi}_{L^{2}(\G\,\vert\,\mathscr C)}^{1/(p+1)}\norm{\Phi^2}_{L^{1,\nu}(\G\,\vert\,\mathscr C)}^{(2p+1)/(p+1)} \\
        &\leq  2CA \norm{\Phi}_{L^{2}_\e(\G\,\vert\,\mathscr C)}^{1/(p+1)}\norm{\Phi}_{\mathfrak E}^{1/(p+1)}\norm{\Phi}_{L^{1,\nu}(\G\,\vert\,\mathscr C)}^{(4p+2)/(p+1)}
    \end{align*} and therefore, taking $D=\max\{2CA,1\}$,  \begin{align*}
        \norm{\Phi^4}_{\mathfrak E}&=\max\{\norm{\Phi^4}_{L^{1,\nu}(\G\,\vert\,\mathscr C)},\norm{\Phi^4}_{L^{\infty}(\G\,\vert\,\mathscr C)}\} \\
        &\leq D\max\{\norm{\Phi}_{L^{2}_\e(\G\,\vert\,\mathscr C)}^{1/(p+1)}\norm{\Phi}_{\mathfrak E}^{1/(p+1)}\norm{\Phi}_{L^{1,\nu}(\G\,\vert\,\mathscr C)}^{(4p+2)/(p+1)}, \norm{\Phi}_{L^2_\e(\G\,\vert\,\mathscr C)}\norm{\Phi^3}_{L^2(\G\,\vert\,\mathscr C)} \} \\
        &\leq D\norm{\Phi}_{L^{2}_\e(\G\,\vert\,\mathscr C)}^{1/(p+1)}\max\{\norm{\Phi}_{\mathfrak E}^{1/(p+1)}\norm{\Phi}_{L^{1,\nu}(\G\,\vert\,\mathscr C)}^{(4p+2)/(p+1)}, \norm{\Phi}_{L^2(\G\,\vert\,\mathscr C)}^{p/(p+1)}\norm{\Phi^3}_{\mathfrak E} \} \\
        &\leq D\norm{\Phi}_{L^{2}_\e(\G\,\vert\,\mathscr C)}^{1/(p+1)}\norm{\Phi}_{\mathfrak E}^{(4p+3)/(p+1)}.
    \end{align*} And now symmetry follows easily: \begin{align*}
        \rho_{\mathfrak E}(\Phi)^4&=\lim_{n\to\infty} \norm{\Phi^{4n}}^{1/n}_{\mathfrak E} \\
        &\leq \lim_{n\to\infty} \big(D\norm{\Phi^n}_{L^{2}_\e(\G\,\vert\,\mathscr C)}^{1/(p+1)}\norm{\Phi^n}_{\mathfrak E}^{(4p+3)/(p+1)}\big)^{1/n} \\
        &\leq \lim_{n\to\infty} \big(D\norm{\lambda(\Phi)^{n-1}}_{\mathbb B(L^{2}_\e(\G\,\vert\,\mathscr C))}^{1/(p+1)}\norm{\Phi}_{L^{2}_\e(\G\,\vert\,\mathscr C)}^{1/(p+1)}\norm{\Phi^n}_{\mathfrak E}^{(4p+3)/(p+1)}\big)^{1/n} \\
        &=\rho_{\mathbb B(L^{2}_\e(\G\,\vert\,\mathscr C))}(\lambda(\Phi))^{1/(p+1)}\rho_{\mathfrak E}(\Phi)^{(4p+3)/(p+1)}
    \end{align*} and thus $\rho_{\mathfrak E}(\Phi)=\rho_{\mathbb B(L^{2}_\e(\G\,\vert\,\mathscr C))}(\lambda(\Phi))$. To conclude, we note that $\lambda(\Phi)\in\mathbb B_a(L^2_\e(\G\,\vert\,\mathscr C))$ and that the inclusion $\mathbb B_a(L^2_\e(\G\,\vert\,\mathscr C))\subset \mathbb B(L^2_\e(\G\,\vert\,\mathscr C))$ is spectral invariant (in particular spectral radius preserving), due to \cite[Corollary 1]{DLW97}. 
\end{proof}

Now we turn our attention to the property of norm-controlled inversion. As mentioned before, the idea is to estimate the $\mathfrak E$-norm of the inverse of an element using both the norms in $\mathfrak E$ and in ${\rm C^*}(\G\,\vert\,\mathscr C)$. We remark that the property of symmetry implies that the operation of taking inverses coincides in both algebras but it does not require or provide any norm estimate for such elements. 

\begin{defn}
    Let $\A\subset\B$ be a continuous inclusion of Banach algebras with the same unit. We say that \emph{$\A$ admits norm-controlled inversion in $\B$} if $\A$ is inverse-closed in $\B$ and there is a function $f:\R^+\times \R^+\to\R^+$ such that $$\norm{a^{-1}}_\A\leq f(\norm{a}_\A,\norm{a^{-1}}_\B).$$ We say that a reduced, symmetric Banach $^*$-algebra \emph{$\A$ admits norm-controlled inversion} if it admits norm-controlled inversion in $\B={\rm C^*}(\A)$. 
\end{defn}

Now we should prove that $\mathfrak E$ admits norm-controlled inversion. We will do so for a very specific case: when $\G$ is discrete and $\mathfrak E$ is unital. We believe that a more general result is possible to obtain, however this is beyond the scope of our subsection, which is dedicated to consequences of our previous results. This result is handled similarly to \cite[Theorem 1.1]{GK13} and its generalization \cite[Proposition 2.2]{SaSh19}.

\begin{prop}
    Let $\theta=\frac{4p+3}{p+1}$ and $D>0$ as in Proposition \ref{gendiff1}. If $\G$ is discrete and $\mathfrak E=\ell^{1,\nu}(\G\,\vert\,\mathscr C)$ is unital, then for every invertible $\Phi\in \mathfrak E$, we have
    \begin{equation}\label{normcontrol}
        \norm{\Phi^{-1}}_{\mathfrak E}\leq \tfrac{\norm{\Phi}_{\mathfrak E}}{\norm{\Phi}^2_{{\rm C^*}(\G\,\vert\,\mathscr C)}}\prod_{k\in\N} \sum_{j=0}^3 \Big(D^{\tfrac{\theta^k-1}{\theta-1}}\big(1-\tfrac{1}{\norm{\Phi^{-1}}^2_{{\rm C^*}(\G\,\vert\,\mathscr C)}\norm{\Phi}^2_{{\rm C^*}(\G\,\vert\,\mathscr C)}}\big)^{4^k-\theta^k}\big(\tfrac{2\norm{\Phi}^2_{\mathfrak E}}{\norm{\Phi}^2_{{\rm C^*}(\G\,\vert\,\mathscr C)}}\big)^{\theta^k}\Big)^j,
    \end{equation}
    with the right hand side being finite. Therefore $\mathfrak E=\ell^{1,\nu}(\G\,\vert\,\mathscr C)$ admits norm-controlled inversion.
\end{prop}
\begin{proof}
    Because of Lemma \ref{lemma}(vi) and Proposition \ref{gendiff1}, we have 
    \begin{align*}
        \norm{\Phi^{4^k}}_{\mathfrak E}&\leq D \norm{\Phi^{4^{k-1}}}_{{\rm C^*}(\G\,\vert\,\mathscr C)}^{1/(p+1)}\norm{\Phi^{4^{k-1}}}_{\mathfrak E}^{(4p+3)/(p+1)} \\
        &\leq D \big(\norm{\Phi}_{{\rm C^*}(\G\,\vert\,\mathscr C)}^{4^{k-1}}\big)^{1/(p+1)}\norm{\Phi^{4^{k-1}}}_{\mathfrak E}^{(4p+3)/(p+1)}
    \end{align*}
    for $k\in\N$. And by letting $\beta_n=\norm{\Phi^{n}}_{\mathfrak E}/\norm{\Phi}^n_{{\rm C^*}(\G\,\vert\,\mathscr C)}$, we obtain the relations 
    $$
    \beta_{4^k}\leq D\beta_{4^{k-1}}^{\theta} \quad\text{ and }\quad \beta_{4^k}\leq D^{{\tfrac{\theta^k-1}{\theta-1}}}\beta_{1}^{\theta^k},
    $$
    so 
    $$
    \norm{\Phi^{4^k}}_{\mathfrak E}\leq D^{\tfrac{\theta^k-1}{\theta-1}}\norm{\Phi}^{4^k}_{{\rm C^*}(\G\,\vert\,\mathscr C)}\big(\tfrac{\norm{\Phi}_{\mathfrak E}}{\norm{\Phi}_{{\rm C^*}(\G\,\vert\,\mathscr C)}}\big)^{\theta^k}=:\alpha_k.
    $$ 
    We now consider the $4$-adic expansions of natural numbers $n\in\N$, written as $n=\sum_{k\in\N} \epsilon_k 4^k$, where $\epsilon_k\in\{0,1,2,3\}$ and only finitely many $\epsilon_k$'s are nonzero. Let us denote by $\mathcal F$ the set of all finitely supported sequences $\epsilon=\{\epsilon_k\}_{k\in\N}\in \{0,1,2,3\}^\N$. Then the $4$-adic expansion is a bijection between $\mathcal F$ and $\N$. We now use the previous bound to get
    $$
    \norm{\Phi^{n}}_{\mathfrak E}=\norm{\prod_{k\in\N}\Phi^{\epsilon_k4^k}}_{\mathfrak E}\leq \prod_{k\in\N}\alpha_k^{\epsilon_k}.
    $$ 
    Summing over $n$, we see that
    \begin{equation}\label{bound}
        \sum_{n\in\N} \norm{\Phi^{n}}_{\mathfrak E}\leq \sum_{\epsilon\in\mathcal F} \prod_{k\in\N}\alpha_k^{\epsilon_k}=\prod_{k\in\N} (1+\alpha_k+\alpha_k^2+\alpha_k^3).
    \end{equation} 
    In fact, one has 
    $$
    \prod_{k=1}^N(1+\alpha_k+\alpha_k^2+\alpha_k^3)=\sum_{\substack{1\leq k_1<\ldots< k_m\leq N\\ \epsilon_1,\ldots,\epsilon_m\in \{0,1,2,3\}}} a_{k_1}^{\epsilon_1}\cdots a_{k_m}^{\epsilon_m}
    $$ 
    and taking the limit $N\to \infty$ yields \eqref{bound}. Also note that the right hand side in \eqref{bound} converges if and only if $\sum_{k\in\N}\alpha_k<\infty$ and this is because 
    $$
    \sum_{k\in\N}\alpha_k\leq\sum_{k\in\N} (\alpha_k+\alpha_k^2+\alpha_k^3)\leq \sum_{k\in\N}\alpha_k+\big(\sum_{k\in\N}\alpha_k\big)^2+\big(\sum_{k\in\N}\alpha_k\big)^3<\infty.
    $$ 
    However, the convergence of $\sum_{k\in\N}\alpha_k$ is easily guaranteed if $\norm{\Phi}_{{\rm C^*}(\G\,\vert\,\mathscr C)}<1$ since $\theta<4$. Now, for a general $\Phi\in \mathfrak E$, invertible in ${\rm C^*}(\G\,\vert\,\mathscr C)$, we proceed as in the proof of \cite[Theorem 3.3]{GK13}, to get 
    $$
    \norm{\Phi^{-1}}_{\mathfrak E}\leq \tfrac{\norm{\Phi}_{\mathfrak E}}{\norm{\Phi}^2_{{\rm C^*}(\G\,\vert\,\mathscr C)}}\prod_{k\in\N} \sum_{j=0}^3 \Big(D^{\tfrac{\theta^k-1}{\theta-1}}\norm{\Psi}^{4^k}_{{\rm C^*}(\G\,\vert\,\mathscr C)}\big(\tfrac{\norm{\Psi}_{\mathfrak E}}{\norm{\Psi}_{{\rm C^*}(\G\,\vert\,\mathscr C)}}\big)^{\theta^k}\Big)^j,
    $$  
    where $\Psi=1-\tfrac{1}{\norm{\Phi}^2_{{\rm C^*}(\G\,\vert\,\mathscr C)}}\Phi^**\Phi$. This element satisfies 
    $$
    \norm{\Psi}_{{\rm C^*}(\G\,\vert\,\mathscr C)}= 1-\tfrac{1}{\norm{\Phi^{-1}}^2_{{\rm C^*}(\G\,\vert\,\mathscr C)}\norm{\Phi}^2_{{\rm C^*}(\G\,\vert\,\mathscr C)}}\quad\text{ and }\quad \norm{\Psi}_{\mathfrak E}\leq 2\tfrac{\norm{\Phi}^2_{\mathfrak E}}{\norm{\Phi}^2_{{\rm C^*}(\G\,\vert\,\mathscr C)}}
    $$ 
    from which the claim follows. The same arguments given show convergence of the right hand side in \eqref{normcontrol}.\end{proof}

\section{Appendix: Some computations of spectral radii}\label{spectralradius}

The idea of this appendix is to use the techniques we developed (as Lemma \ref{lemma}) to compute the spectral radius in many cases of interest. In particular, we obtain two things: symmetry of the $L^1$-algebras $L^1(\G\,\vert\,\mathscr C)$ associated with compact groups and `quasi-symmetry', in the sense of \cite{SaWi20}, for more general algebras (associated to groups of subexponential growth).

The first result deals with groups of subexponential growth. We recall the definition. 
\begin{defn}
    A locally compact group $\G$ is said to have \emph{subexponential growth} if 
    $$
    \lim_{n\to\infty}\mu(K^n)^{1/n}=1,
    $$ 
    for all relatively compact subsets $K\subset \G$.
\end{defn}

\begin{rem}
    It is well-known that groups of subexponential growth are unimodular \cite[Proposition 12.5.8]{Pa94} and amenable. Their amenability also follows from Proposition \ref{subexp}.
\end{rem}

This proposition vastly generalizes the previous best result -valid only for the trivial line bundle- and implies, in the terminology of \cite{SaWi20}, that $C_{\rm c}(\G\,\vert\,\mathscr C)$ is `quasi-symmetric' in $L^1(\G\,\vert\,\mathscr C)$.

\begin{prop}\label{subexp}
    Let $\G$ be a group of subexponential growth. Then for all $\Phi\in L^\infty(\G\,\vert\,\mathscr C)$ and of compact support, \begin{equation*}
        {\rm Spec}_{L^1(\G\,\vert\,\mathscr C)}(\Phi)={\rm Spec}_{{\rm C^*}(\G\,\vert\,\mathscr C)}(\lambda(\Phi)).
    \end{equation*} \end{prop}
\begin{proof}
    Due to \cite[Theorem 2.3]{SaWi20}, it is enough to show equality for the corresponding spectral radii. It is clear that $\rho_{{\rm C^*}(\G\,\vert\,\mathscr C)}(\lambda(\Phi))\leq\rho_{L^1(\G\,\vert\,\mathscr C)}(\Phi)$. In order to show the other inequality, we let $\Phi$ be as in the hypothesis of the proposition and see that 
    $$
    \norm{\Phi^n}_{L^\infty(\G\,\vert\,\mathscr C)}\leq  \norm{\Phi}_{L^2(\G\,\vert\,\mathscr C)}\,\norm{\Phi^{n-1}}_{L^2_\e(\G\,\vert\,\mathscr C)}\leq\norm{\Phi}_{L^2(\G\,\vert\,\mathscr C)}\,\norm{\lambda(\Phi)^{n-2}}_{\mathbb B(L^2_\e(\G\,\vert\,\mathscr C))}\,\norm{\Phi}_{L^2_\e(\G\,\vert\,\mathscr C)}. 
    $$
    Now, H\"older's inequality allows us to link the $L^\infty$-norm with the spectral radius as follows: 
    \begin{align*}
        \rho_{L^1(\G\,\vert\,\mathscr C)}(\Phi)&=\lim_{n\to\infty} \norm{\Phi^n}_{L^1(\G\,\vert\,\mathscr C)}^{1/n} \\
        &\leq \lim_{n\to\infty}\norm{\Phi^n}_{L^\infty(\G\,\vert\,\mathscr C)}^{1/n}\mu\big({\rm Supp}(\Phi^n)\big)^{1/n} \\
        &\leq \lim_{n\to\infty}\norm{\Phi}_{L^2(\G\,\vert\,\mathscr C)}^{2/n}\,\norm{\lambda(\Phi)^{n-2}}_{\mathbb B(L^2_\e(\G\,\vert\,\mathscr C))}^{1/n}\mu\big({\rm Supp}(\Phi)^n\big)^{1/n} \\
        &=\lim_{n\to\infty} \norm{\lambda(\Phi)^{n-2}}_{\mathbb B(L^2_\e(\G\,\vert\,\mathscr C))}^{1/n} \\
        &=\rho_{\mathbb B(L^2_\e(\G\,\vert\,\mathscr C))}(\lambda(\Phi)).
    \end{align*} 
    To conclude, we note that $\lambda(\Phi)\in\mathbb B_a(L^2_\e(\G\,\vert\,\mathscr C))$ and that the inclusion $\mathbb B_a(L^2_\e(\G\,\vert\,\mathscr C))\subset \mathbb B(L^2_\e(\G\,\vert\,\mathscr C))$ is spectral invariant, because of \cite[Corollary 1]{DLW97}. 
\end{proof}

Informally, the previous result states that, under subexponential growth conditions, the spectra of a continuous function with compact support is independent of the algebra of reference. Therefore, following this hint -and in the hope of a stronger result-, we turn our attention to continuous functions over compact groups.

\begin{lem}\label{reduction}
    Let $p\geq q$ with $p,q\in[1,\infty]$ and suppose that $\G$ is compact. Then $L^p(\G\,\vert\,\mathscr C)$ is inverse closed in $L^q(\G\,\vert\,\mathscr C)$. Moreover, $L^p(\G\,\vert\,\mathscr C)$ is symmetric if and only if $L^1(\G\,\vert\,\mathscr C)$ is symmetric and this happens if and only if $C(\G\,\vert\,\mathscr C)$ is symmetric.
\end{lem}
\begin{proof}
    Young's inequality gives $\norm{\Phi*\Psi}_{L^p(\G\,\vert\,\mathscr C)}\leq \norm{\Phi}_{L^q(\G\,\vert\,\mathscr C)}\norm{\Psi}_{L^p(\G\,\vert\,\mathscr C)}$. Therefore for $n\in\N$ and $\Phi\in L^p(\G\,\vert\,\mathscr C)$ one has $$\norm{\Phi^{n}}_{L^p(\G\,\vert\,\mathscr C)}\leq \norm{\Phi^{n-1}}_{L^q(\G\,\vert\,\mathscr C)}\norm{\Phi}_{L^p(\G\,\vert\,\mathscr C)},$$ so $$\rho_{L^p(\G\,\vert\,\mathscr C)}(\Phi)=\lim_{n\to\infty} \norm{\Phi^{n}}_{L^p(\G\,\vert\,\mathscr C)}^{1/n}\leq \lim_{n\to\infty} \norm{\Phi^{n-1}}_{L^q(\G\,\vert\,\mathscr C)}^{1/n}\norm{\Phi}_{L^p(\G\,\vert\,\mathscr C)}^{1/n}= \rho_{L^q(\G\,\vert\,\mathscr C)}(\Phi),$$ hence $\rho_{L^p(\G\,\vert\,\mathscr C)}(\Phi)=\rho_{L^q(\G\,\vert\,\mathscr C)}(\Phi)$, as the reversed inequality always holds. Thus inverse-closeness follows from Lemma \ref{barnes}. 
    
    For the second statement: $C(\G\,\vert\,\mathscr C)$ is a closed $^*$-subalgebra of $L^\infty(\G\,\vert\,\mathscr C)$, so symmetry of the latter implies symmetry of the former \cite[Theorem 11.4.2]{Pa94}. On the other hand, if $L^p(\G\,\vert\,\mathscr C)$ is assumed symmetric, it becomes a symmetric dense, two-sided $^*$-ideal of $L^1(\G\,\vert\,\mathscr C)$ and the conclusion follows from \cite[page 86]{Wi78}. The same reasoning applies to $C(\G\,\vert\,\mathscr C)$.
\end{proof}

In particular, and with the help of the reduction done in Lemma \ref{reduction}, we obtain the symmetry of all $L^p$-algebras of compact groups.
\begin{thm}\label{corcompact}
    Suppose $\G$ is compact and let $p\in [1,\infty]$. Then, for $\B=L^p(\G\,\vert\,\mathscr C)$ or $\B=C(\G\,\vert\,\mathscr C)$, $\B$ is symmetric and inverse-closed in ${\rm C^*}(\G\,\vert\,\mathscr C)$.
\end{thm}
\begin{proof}
    Lemma \ref{reduction} and its proof show that for all $\Phi\in L^\infty(\G\,\vert\,\mathscr C)$, we have the equality of spectra $${\rm Spec}_{L^1(\G\,\vert\,\mathscr C)}(\Phi)={\rm Spec}_{L^\infty(\G\,\vert\,\mathscr C)}(\lambda(\Phi)).$$ Then the result follows from Proposition \ref{subexp} and another application of Lemma \ref{reduction}.
\end{proof}

\begin{rem}
    We remark that the best result of this sort previously available in the literature only considered algebras arising from (untwisted) actions \cite[Theorem 1]{LP79}. Their method is completely different from ours, as they embed the algebras into bigger ones and deal with multipliers. Moreover, it seems reasonable to argue that our method is simpler. Their result, however, still works when the algebra of coefficients is not a $C^*$-algebra but a symmetric Banach $^*$-algebra.
\end{rem}

\begin{rem}
        A consequence of Theorem \ref{corcompact} is that $L^1(\G\,\vert\,\mathscr C)$ is symmetric as soon as $\G$ is compact. In the terminology of \cite{FJM}, this means that compact groups are hypersymmetric. We have thus provided the first examples of hypersymmetric groups with non-symmetric discretizations. Consider, for example, the rotation group ${\rm SO}(3)$, which is compact but its discretization is not even symmetric since it contains a free subgroup (see \cite{Je70}).
\end{rem}

\section*{Acknowledgments}

The author thankfully acknowledges support by the NSF grant DMS-2000105. He is also grateful to Professor Ben Hayes for his very helpful comments and to the anonymous referee, who suggested many improvements.

\printbibliography

\bigskip
\bigskip
ADDRESS

\smallskip
Felipe I. Flores

Department of Mathematics, University of Virginia,

114 Kerchof Hall. 141 Cabell Dr,

Charlottesville, Virginia, United States

E-mail: hmy3tf@virginia.edu

\end{document}